\documentclass{amsart}

\usepackage[english]{babel}

\usepackage{amstext,amsmath,amsthm,amssymb,mathrsfs}
\usepackage{float}

\usepackage{txfonts}
\usepackage{enumerate}
\usepackage[T1]{fontenc}

\usepackage[small,nohug,heads=littlevee]{diagrams}
\diagramstyle[labelstyle=\scriptstyle]

\newcommand{\Longlra}{\ensuremath{\Longleftrightarrow}}

\newcommand{\longra}{\ensuremath{\longrightarrow}}

\newcommand{\B}{\varmathbb{B}}
\newcommand{\C}{\varmathbb{C}}

\newcommand{\R}{\varmathbb{R}}
\newcommand{\M}{\varmathbb{M}}
\newcommand{\N}{\varmathbb{N}}
\newcommand{\F}{\varmathbb{F}}

\newcommand{\E}{\varmathbb{E}}

\newcommand{\norm}[1]{||#1||}

\DeclareMathAlphabet{\mathpzc}{OT1}{pzc}{m}{it}

\DeclareMathOperator{\supp}{supp}

\theoremstyle{plain}
\newtheorem{thm}{Theorem}[section]
\newtheorem{lemma}[thm]{Lemma}
\newtheorem{prop}[thm]{Proposition}
\newtheorem{cor}[thm]{Corollary}

\theoremstyle{definition}

\newtheorem{ex}[thm]{Example}

\theoremstyle{remark}
\newtheorem{remark}[thm]{Remark}

\usepackage[colorinlistoftodos,prependcaption,textsize=tiny]{todonotes}

\newcommand{\tr}{\text{Tr}}

\newcommand{\ext}{{\rm ext}}

\newcommand{\Dir}{{\rm Dir}}

\newcommand{\dom}{\mathscr{O}}

\begin{document}

\title[Second Order Parabolic Problems subject Dirichlet Boundary Conditions]{Second Order Operators Subject to Dirichlet Boundary Conditions in Weighted Triebel-Lizorkin Spaces: Parabolic Problems}

\author{Nick Lindemulder}
\address{Institute of Analysis \\
Karlsruhe Institute of Technology \\
Englerstra\ss e 2 \\
76131 Karlsruhe\\
Germany}
\email{nick.lindemulder@kit.edu}


\thanks{The author was supported by the Vidi subsidy 639.032.427 of the Netherlands Organisation for Scientific Research (NWO)}

\date{\today}

\begin{abstract}
In this paper we consider second order parabolic partial differential equations subject to the Dirichlet boundary condition on smooth domains. We establish weighted $L_{q}$-maximal regularity in weighted Triebel-Lizorkin spaces for such parabolic problems with inhomogeneous boundary data.
The weights that we consider are power weights in time and space, and yield flexibility in the optimal regularity of the initial-boundary data, allow to avoid compatibility conditions at the boundary and provide a smoothing effect. In particular, we can treat rough inhomogeneous boundary data.
\end{abstract}

\maketitle

\section{Introduction}

This paper is a continuation of \cite{Lindemulder2018_elliptic_weighted_B&F}, where the boundedness of the $H^{\infty}$-calculus for the Dirichlet realizations of second order elliptic operators in weighted Triebel-Lizorkin spaces was established.
Since the work \cite{Kalton&Weis2001_sums_of_closed_operators} the $H^\infty$-calculus has become an extremely powerful tool in the theory of PDEs (see the monographs \cite{Denk&Kaip2013,Pruess&Simonett2016_book} and references therein).
For this paper the interest lies in the fact that the $H^{\infty}$-calculus can be used to obtain $L_{q}$-maximal regularity.
But the $H^{\infty}$-calculus can also be used for many other things (see \cite{Hytonen&Neerven&Veraar&Weis2016_Analyis_in_Banach_Spaces_II,Weis2006_survey}).

Maximal regularity is an important tool in the theory of nonlinear PDEs.
Maximal regularity means that there is an isomorphism between the data and the solution of the problem in suitable function spaces.
Having established maximal regularity for the linearized problem, the nonlinear problem can be treated with tools as the contraction principle and the implicit function theorem.
Let us mention \cite{Angenent1990,Clement&Simonett2001} for approaches in spaces of continuous functions,
\cite{Acquistapace&Terreni,Lunardi_book1995} for approaches in H\"older spaces and
\cite{Amann1995_Lin_and_quasilinear_parabolic,Amann_Maximal_regularity_and_quasilinear_pbvps_problems,Clement_Li,
Clement&Pruss1992,EPS,Pruss2002,Pruess&Simonett2016_book} for approaches in $L_{p}$-spaces (with $p \in (1,\infty)$).

In order to describe the specific maximal regularity result that is obtained in this paper, let us consider the heat equation
\begin{equation}\label{DSOP:eq:heat_eq}
\left\{\begin{array}{rll}
\partial_{t}u - \Delta u &= f &\quad\text{on}\quad J \times \mathscr{O},\\
u_{|\partial\mathscr{O}} &= g &\quad\text{on}\quad J \times \partial\mathscr{O},\\
u(0) &= u_{0} &\quad\text{on}\quad \mathscr{O},
\end{array}\right.
\end{equation}
where $J=(0,T)$ with $T \in (0,\infty]$ and where $\mathscr{O}$ is a smooth domain in $\R^{d}$ with a compact boundary $\partial\mathscr{O}$.

Let $q \in (1,\infty)$, $\mu \in (-1,q-1)$ and set $v_{\mu}(t):=t^{\mu}$. Let $\E \subset \mathcal{D}'(\mathscr{O})$ be a Banach space of distributions on $\mathscr{O}$ such that there exists a notion of trace on the associated second order space $\E^{2} = \{ u \in  \mathcal{D}(\mathscr{O}) : D^{\alpha}u \in \E, |\alpha| \leq 2 \}$ that is described by a bounded linear operator $\mathrm{Tr}_{\partial\mathscr{O}}:\E^{2} \longra \F$ for some suitable Banach space.
The following choices for $\E$ are of interest for this paper:
\begin{enumerate}[(i)]
\item\label{DSOP:it:intro:case1} $\E = L_{p}(\mathscr{O},w_{\gamma}^{\partial\mathscr{O}})$ with $p \in (1,\infty)$ and $\gamma \in (-1,2p-1)$,
\item\label{DSOP:it:intro:case2} $\E = F^{s}_{p,r}(\mathscr{O},w_{\gamma}^{\partial\mathscr{O}})$ with $p,r \in (1,\infty)$, $\gamma \in (-1,\infty)$ and $s \in (\frac{1+\gamma}{p}-2,\frac{1+\gamma}{p})$,
\end{enumerate}
where $w_{\gamma}^{\partial\mathscr{O}}=\mathrm{dist}(x,\partial\mathscr{O})^{\gamma}$; in both cases one can take $\F=L_{p}(\partial\mathscr{O})$. For these choices we speak of $L_{q,\mu}$-$L_{p,\gamma}$-maximal regularity and $L_{q,\mu}$-$F^{s}_{p,r,\gamma}$-maximal regularity, respectively.
There actually is an overlap between case \eqref{DSOP:it:intro:case1} and \eqref{DSOP:it:intro:case2}, see \eqref{DSOP:eq:intro:case1&2} below.

In the $L_{q,\mu}$-$\E$-maximal regularity approach to \eqref{DSOP:eq:heat_eq} one is looking for solutions $u$ in the \emph{maximal regularity space}
\begin{equation}\label{DSOP:eq:max-reg_space}
W^{1}_{q}(J,v_{\mu};\E) \cap L_{q}(J,v_{\mu};\E^{2}),
\end{equation}
where the boundary condition $u_{|\partial\mathscr{O}} = g$ has to be interpreted as $\mathrm{Tr}_{\partial\mathscr{O}}u=g$.
The problem \eqref{DSOP:eq:heat_eq} is said to enjoy the property of \emph{maximal $L_{q,\mu}$-$\E$-regularity} if there exists a (necessarily unique) space of initial-boundary data $\mathscr{D}_{i.b.} \subset L_{q}(J,v_{\mu};\F) \times \E$ such that for every $f \in L_{q}(J,v_{\mu};\E)$ it holds that \eqref{DSOP:eq:heat_eq} has a unique solution $u$ in \eqref{DSOP:eq:max-reg_space}
if and only if $(g,u_{0}) \in \mathscr{D}_{i.b.}$.
In this situation there exists a Banach norm on $\mathscr{D}_{i.b.}$, unique up to equivalence, with
\[
\mathscr{D}_{i.b.} \hookrightarrow L_{q}(J,v_{\mu};\F) \oplus \E,
\]
which makes the associated solution operator a topological linear isomorphism between the data space $L_{q}(J,v_{\mu};\E)) \oplus \mathscr{D}_{i.b.}$ and the solution space $W^{1}_{q}(J,v_{\mu};\E) \cap L_{q}(J,v_{\mu};\E^{2})$.
The \emph{maximal $L_{q,\mu}$-$\E$-regularity problem} for \eqref{DSOP:eq:heat_eq} consists of establishing maximal $L_{q,\mu}$-$\E$-regularity
for \eqref{DSOP:eq:heat_eq} and explicitly determining the space $\mathscr{D}_{i.b.}$.

The $L_{q,\mu}$-$L_{p,\gamma}$-maximal regularity problem for \eqref{DSOP:eq:heat_eq} was recently solved in \cite{LV2018_Dir_Laplace}. In the special case $f=0$, $u_{0}=0$ and $q=p$, $\mu=0$, one gets the following result.

\begin{thm}(\cite[Theorem~1.2]{LV2018_Dir_Laplace})\label{DSOP:thm:LV2018_Dir_Laplace}
Let $\mathscr{O}$ be a bounded $C^{2}$-domain in $\R^{d}$ and let $J=(0,T)$ with $T \in (0,\infty]$.
Let $p \in (1,\infty)$, $\gamma \in (-1,2p-1) \setminus \{p-1,2p-3\})$ and set $\delta = 1-\frac{1+\gamma}{2p} \neq \frac{1}{p}$. Then \eqref{DSOP:eq:heat_eq} enjoys the property of $L_{p}$-$L_{p,\gamma}$-maximal regularity and $u \mapsto u_{|\partial\mathscr{O}}$ (in the sense of traces) defines an isomorphism from
\[
\left\{ u \in W^{1}_{p}(J;L_p(\dom,w_{\gamma}^{\partial\dom}))\cap L_{p}(J;W^{2,p}(\dom,w_{\gamma}^{\partial\dom})) : \partial_{t}u-\Delta u = 0, u(0)=0 \right\}
\]
to
\[
{_{0}}B^{\delta}_{p,p}(J;L_{p}(\partial\dom)) \cap
L_{p}(J;B^{2\delta}_{p,p}(\partial\dom)),
\]
where
\[
{_{0}}B^{\delta}_{p,p}(J;X) = \left\{\begin{array}{ll}
B^{\delta}_{p,p}(J;X),& 0< \delta < \frac{1}{p},\\
\{ v \in B^{\delta}_{p,p}(J;X) : v(0)=0 \},& \frac{1}{p} < \delta < 1.\\
\end{array}\right.
\]
\end{thm}

Note that $\delta \in (0,1)$ can be taken arbitrarily close to $0$ by taking $\gamma$ arbitrarily close to $2p-1$, whereas $\delta \in (\frac{1}{2},1)$ for $\gamma$ in the $A_{p}$-range $(-1,p-1)$.

For $q \neq p$ the Besov space $B^{\delta}_{p,p}$ (in time) needs to be replaced by the Triebel-Lizorkin space $F^{\delta}_{q,p}$; moreover, the condition $\gamma \neq 2p-3$ turns into $\gamma \neq 2p-\frac{2p}{q}-1$ and the critical value $\frac{1}{p}$ turns into $\frac{1}{q}$ in the corresponding definition for ${_{0}}F^{\delta}_{q,p}$.

Earlier works on the (weighted) $L_{q}$-$L_{p}$-maximal regularity problem for parabolic initial-boundary value problems with inhomogeneous date include \cite{DHP2,Lindemulder2017_max-reg,LSU,MeySchnau2,Weidemaier_intro}, where \cite{LSU} ($q=p$, $\mu=\gamma=0$) and \cite{Weidemaier_intro} ($p \leq q$, $\mu=\gamma=0$) are on scalar-valued 2nd order problems with Dirichlet and Neumann boundary conditions and where \cite{DHP2} ($\mu=\gamma=0$), \cite{MeySchnau2} ($q=p$, $\mu \in [0,q-1)$, $\gamma=0$) and \cite{Lindemulder2017_max-reg} ($\mu \in (-1,q-1)$, $\gamma \in (-1,p-1)$) are on higher order operators with Lopatinskii-Shapiro boundary conditions in a Banach space-valued stetting (generalizing finite dimensional systems).

The contribution of Theorem~\ref{DSOP:thm:LV2018_Dir_Laplace} is the treatment of the case $\gamma \in (p-1,2p-1)$.
The restriction $\gamma \in (-1,p-1)$  for the spatial weight $w^{\partial\mathscr{O}}_{\gamma}$ in \cite{Lindemulder2017_max-reg} is a restriction of harmonic analytic nature. Indeed, $(-1,p-1)$ is the Muckenhoupt $A_{p}$-range for $w^{\partial\mathscr{O}}_{\gamma}$: given $p \in (1,\infty)$ and $\gamma \in \R$, it holds that
\begin{equation}\label{Boutet:eq:intro:Ap-cond_power-weight}
w_{\gamma}^{\partial\mathscr{O}} = \mathrm{dist}(\,\cdot\,,\partial\mathscr{O})^{\gamma} \in A_{p}(\R^{n}) \quad \Longlra \quad \gamma \in (-1,p-1).
\end{equation}
The Muckenhoupt class $A_{p}(\R^{n})$ ($p \in (1,\infty)$) is a class of weights for which many harmonic analytic tools from the unweighted setting, such as Mikhlin Fourier multiplier theorems and Littlewood-Paley decompositions, remain valid for the corresponding weighted $L_{p}$-spaces.
For example, the Littlewood-Paley decomposition for $L_{p}(\R^{n},w)$ with $w \in A_{p}(\R^{n})$ and its variant for $W^{k}_{p}(\R^{n},w)$, $k \in \N$, can be formulated by means of Triebel-Lizorkin spaces as
\[
L_{p}(\R^{n},w)= F^{0}_{p,2}(\R^{n},w), \qquad W^{k}_{p}(\R^{n},w)= F^{k}_{p,2}(\R^{n},w),
\]
As a consequence of this Littlewood-Paley decomposition and \eqref{Boutet:eq:intro:Ap-cond_power-weight},
\begin{equation}\label{DSOP:eq:intro:case1&2}
L_{p}(\mathscr{O},w_{\gamma}^{\partial\mathscr{O}}) = F^{0}_{p,2}(\mathscr{O},w_{\gamma}^{\partial\mathscr{O}}), \qquad \gamma \in (-1,p-1).
\end{equation}
The main difficulty in \cite{LV2018_Dir_Laplace} in the non-$A_p$ range $(p-1,2p-1)$ is that these standard tools are no longer available.

One way to avoid these difficulties is to work in the weighted Triebel-Lizorkin spaces $\E = \E=F^{s}_{p,r}(\mathscr{O},w_{\gamma}^{\partial\mathscr{O}})$ instead of $\E = L_{p}(\mathscr{O},w_{\gamma}^{\partial\mathscr{O}})$. The advantage of the scale of weighted Triebel-Lizorkin spaces is the strong harmonic analytic nature of these function spaces, leading to the availability of many powerful tools (see e.g.\ \cite{Bui&Paluszynski&Taibleson1996,Bui1982,Bui1994,Haroske&Piotrowska2008,Haroske&Skrzypczak2008_EntropyI,
Haroske&Skrzypczak2011_EntropyII,Haroske&Skrzypczak2011_EntropyIII,Lindemulder2018_elliptic_weighted_B&F,
Meyries&Veraar2014_char_class_embeddings,
Meyries&Veraar2012_sharp_embedding,Meyries&Veraar2014_traces,Sickel&Skrzypczak&Vybiral2014}). In particular, there is a Mikhlin-H\"ormander Fourier multiplier theorem.
That Mikhlin-H\"ormander Fourier multiplier theorem which was used in \cite{Lindemulder2018_elliptic_weighted_B&F} on the boundedness of the $H^{\infty}$-calculus for the Dirichlet realizations of second order elliptic operators in weighted Triebel-Lizorkin spaces, serving as a basis for the present paper.

The main result of this paper, Theorem~\ref{DSOP:thm:main;half-line}, provides a solution to the $L_{q,\mu}$-$F^{s}_{p,r,\gamma}$-maximal regularity problem for second order parabolic problems subject to inhomogeneous Dirichlet boundary conditions.
For purposes of exposition, without loosing the main innovative part of the present paper, let us state Theorem~\ref{DSOP:thm:main;half-line} (also see Remark~\ref{DSOP:rmk:thm:main;half-line;lambda0}) for the special case of the heat equation \eqref{DSOP:eq:heat_eq} with $f=0$, $u_{0}=0$ and $q=p$, $\mu=0$.

\begin{thm}\label{DSOP:thm:main_intro}
Let $\mathscr{O}$ be a $C^{\infty}$-domain in $\R^{d}$ with a compact boundary $\partial\dom$ and let $J=(0,T)$ with $T \in (0,\infty]$.
Let $p,r \in (1,\infty)$, $\gamma \in (-1,\infty)$, $s \in (\frac{1+\gamma}{p}-2,\frac{1+\gamma}{p})$ and set $\delta = \frac{s}{2}+1-\frac{1+\gamma}{2p}$. If $\delta \neq \frac{1}{p}$, then \eqref{DSOP:eq:heat_eq} enjoys the property of $L_{p}$-$F^{s}_{p,r,\gamma}$-maximal regularity and $u \mapsto u_{|\partial\mathscr{O}}$ (in the sense of traces) defines an isomorphism from
\[
\left\{ u \in W^{1}_{p}(J;F^{s}_{p,r}(\dom,w_{\gamma}^{\partial\dom}))\cap L_{p}(J;F^{s+2}_{p,r}(\dom,w_{\gamma}^{\partial\dom})) : \partial_{t}u-\Delta u = 0, u(0)=0 \right\}
\]
to
\begin{equation}\label{DSOP:eq:thm:main_intro;bd_space}
{_{0}}B^{\delta}_{p,p}(J;L_{p}(\partial\dom)) \cap
L_{p}(J;B^{2\delta}_{p,p}(\partial\dom)).
\end{equation}
\end{thm}

Note that $\delta = \frac{s}{2}+1-\frac{1+\gamma}{2p}$ can take any value in $(0,1)$ by choosing $\gamma$ and $s$ appropriately. In fact, given $\delta \in (0,1)$ and $\gamma \in (-1,\infty)$, choosing $s=2\delta+\frac{1+\gamma}{p}-2 \in (\frac{1+\gamma}{p}-2,\frac{1+\gamma}{p})$ gives $\delta = \frac{s}{2}+1-\frac{1+\gamma}{2p}$.

Playing around with the weight parameter $\gamma$ and the smoothness parameter $s$, sharp Sobolev embedding for Triebel-Lizorkin spaces from \cite{Meyries&Veraar2012_sharp_embedding} and an elementary embedding of Triebel-Lizorkin spaces into Sobolev spaces yield the following self-improvement in the setting of the above theorem (see Corollary~\ref{DSOP:cor:thm:main;half-line;smoothing} and Remark~\ref{DSOP:rmk:cor:thm:main;half-line;smoothing}).

\begin{cor}\label{DSOP:cor:thm:main_intro;smoothing_F-scale}
Let the notation and assumptions be as in Theorem~\ref{DSOP:thm:main_intro}. Then
\begin{align}
&\left\{ u \in W^{1}_{p}(J;F^{s}_{p,r}(\dom,w_{\gamma}^{\partial\dom}))\cap L_{p}(J;F^{s+2}_{p,r}(\dom,w_{\gamma}^{\partial\dom})) : \partial_{t}u-\Delta u = 0, u(0)=0 \right\} \nonumber \\
&\qquad \hookrightarrow
\quad \bigcap_{\nu > -1}\left[ W^{1}_{p}(J;F^{s+\frac{\nu-\gamma}{p}}_{p,1}(\dom,w_{\nu}^{\partial\dom}))\cap L_{p}(J;F^{s+\frac{\nu-\gamma}{p}+2}_{p,1}(\dom,w_{\nu}^{\partial\dom})) \right] \label{DSOP:eq:cor:thm:main_intro;smoothing_F-scale;1} \\
&\qquad\qquad \hookrightarrow \quad \bigcap_{k \in \N}\left[ W^{1}_{p}(J;W^{k}_{p}(\dom,w_{\gamma+(k-s)p}^{\partial\dom}))\cap L_{p}(J;W^{k+2}_{p}(\dom,w_{\gamma+(k-s)p}^{\partial\dom})) \right]. \label{DSOP:eq:cor:thm:main_intro;smoothing_F-scale;2}
\end{align}
\end{cor}

The crucial feature of the Sobolev embedding from \cite{Meyries&Veraar2012_sharp_embedding},
\begin{equation}\label{DSOP:eq:Sob_embd_intro}
F^{s_{0}}_{p,r_{0}}(\dom,w^{\partial\dom}_{\gamma_{0}}) \hookrightarrow F^{s_{1}}_{p,r_{1}}(\dom,w^{\partial\dom}_{\gamma_{1}}), \qquad \gamma_{1} > \gamma_{0}, s_{0} = s_{1}+\frac{\gamma_{0}-\gamma_{1}}{p},
\end{equation}
on which the self-improvement \eqref{DSOP:eq:cor:thm:main_intro;smoothing_F-scale;1} is based is the invariance of the sharp space of boundary data \eqref{DSOP:eq:thm:main_intro;bd_space} under this embedding:
\[
\delta_{s_{0},p,\gamma_{0}} = \frac{s_{0}}{2}+1-\frac{1+\gamma_{0}}{2p} = \frac{s_{1}}{2}+1-\frac{1+\gamma_{1}}{2p} = \delta_{s_{1},p,\gamma} \qquad \text{when}\quad
s_{0} = s_{1}+\frac{\gamma_{0}-\gamma_{1}}{p}.
\]

The embedding \eqref{DSOP:eq:cor:thm:main_intro;smoothing_F-scale;2} follows directly from the elementary embedding
\[
F^{k}_{p,1}(\dom,w^{\partial\dom}_{\nu})
\hookrightarrow W^{k}_{p}(\dom,w^{\partial\dom}_{\nu}), \qquad k \in \N, \nu > -1,
\]
by restriction to $\nu \in (-1,\infty)$ for which $s+\frac{\nu-\gamma}{p} = k \in \N$ (so that $\nu=\gamma+(k-s)p$).

The spaces \eqref{DSOP:eq:cor:thm:main_intro;smoothing_F-scale;1} and \eqref{DSOP:eq:cor:thm:main_intro;smoothing_F-scale;2} are spaces of function that are smooth in the spatial variable, which can be seen by Sobolev embedding.
The varying weight parameter $\gamma+(k-s)p$ in $k$ can be seen as a way to measure the allowed blow-up of the spatial derivatives. Introducing the function spaces
\[
\mathcal{W}^{\infty,\ell}_{p,\gamma}(\dom) := \left\{ v \in W^{\ell}_{p}(\dom,w^{\partial\dom}_{\gamma}) : D^{\alpha}v  \in L_{p}(\dom,w^{\partial\dom}_{\gamma +(|\alpha|-\ell)p}),\, |\alpha| > \ell  \right\} \hookrightarrow C^{\infty}(\dom),
\]
where $\ell \in \N$, $p \in (1,\infty)$ and $\gamma \in (-1,\infty)$, this feature is emphasized by the trivial observation that the space \eqref{DSOP:eq:cor:thm:main_intro;smoothing_F-scale;2} is continuously embedded into
\begin{equation}\label{DSOP:eq:cor:thm:main_intro;smoothing_F-scale;3}
W^{1}_{p}(J;\mathcal{W}^{\infty,0}_{p,\gamma-sp}(\dom)) \cap L_{p}(J;\mathcal{W}^{\infty,2}_{p,\gamma-sp}(\dom)),
\end{equation}
which can be viewed as the space of all $u \in C(\overline{J};C^{\infty}(\dom))$ with $\partial_{t}D^{\alpha}u \in L_{p}(J;L_{p}(\dom,w^{\partial\dom}_{\gamma-sp +|\alpha|p}))$ and $D^{\alpha}u \in L_{p}(J;L_{p}(\dom,w^{\partial\dom}_{\gamma-sp +\max\{|\alpha|-2,1\}p}))$ for each $\alpha \in \N^{d}$.
Now note that the contribution of $w_{\nu}^{\partial\dom}$ in the defining integral of the $L_{p}(\dom,w_{\nu}^{\partial\dom})$-norm gets less as $\nu$ increases, allowing functions in $L_{p}(\dom,w_{\nu}^{\partial\dom})$ to have more blow-up near the boundary $\partial\dom$ for bigger $\nu$. So we have $C^{\infty}$-smoothness in the space-variable with some quantified blow-up behaviour near the boundary $\partial\dom$ in terms of the derivatives.

Putting Theorem~\ref{DSOP:thm:main_intro} and Corollary~\ref{DSOP:cor:thm:main_intro;smoothing_F-scale} together with Theorem~\ref{DSOP:thm:LV2018_Dir_Laplace} we get a smoothing effect for that result as well:
\begin{cor}\label{DSOP:cor:thm:main_intro;smoothing_W-scale}
Let the notation and assumptions be as in Theorem~\ref{DSOP:thm:LV2018_Dir_Laplace}. Then
\begin{align*}
&\left\{ u \in W^{1}_{p}(J;L^p(\dom,w_{\gamma}^{\partial\dom}))\cap L_{p}(J;W^{2,p}(\dom,w_{\gamma}^{\partial\dom})) : \partial_{t}u-\Delta u = 0, u(0)=0 \right\} \\
&\qquad\qquad \hookrightarrow \quad \bigcap_{k \in \N}\left[ W^{1}_{p}(J;W^{k}_{p}(\dom,w_{\gamma+kp}^{\partial\dom}))\cap L_{p}(J;W^{k+2}_{p}(\dom,w_{\gamma+kp}^{\partial\dom})) \right].
\end{align*}
\end{cor}

The idea to use scales of weighted Triebel-Lizorkin spaces to obtain smoothing in a (weighted) $L_{p}$-setting also turn out to be useful in \cite{LV2018_Boundary_noise} on the heat equation with multiplicative Dirichlet boundary noise.

\textbf{Organization of the paper.}
The paper is organized as follows. Section~\ref{DSOP:sec:prelim} is devoted to the necessary preliminaries, Section~\ref{DSOP:sec:trace} characterizes the spatial trace space of the maximal regularity space and
Section~\ref{DSOP:sec:max-reg} provides a solution to the $L_{q,\mu}$-$F^{s}_{p,r,\gamma}$-maximal regularity problem for second order parabolic problems with inhomogeneous Dirichlet boundary conditions.

\textbf{Notations and conventions.} All vector spaces are over the field of complex scalars $\C$.

\section{Preliminaries}\label{DSOP:sec:prelim}

\subsection{Muckenhoupt Weights}\label{DSOP:subsec:prelim:weights}

A reference for the general theory of Muckenhoupt weights is \cite[Chapter~9]{Grafakos_modern}.

A \emph{weight} on a measure space $(S,\mathscr{A},\mu)$ is a measurable function $w:S \longra [0,\infty]$ that takes it values almost everywhere in $(0,\infty)$.
Let $w$ be a weight on $S$. For $p \in [1,\infty)$ we denote by $L_{p}(S,w)$ the space of all equivalence classes of measurable functions $f:S \longra \C$ with
\[
\norm{f}_{L^{p}(S,w)} := \left( \int_{S}|f(x)|^{p}w(x)\,d\mu(x) \right)^{1/p} < \infty.
\]
If $p \in (1,\infty)$, then $w'=w'_{p} := w^{-\frac{1}{p-1}}$ is also a weight on $S$, called the $p$-dual weight of $w$.
Furthermore, for $p \in (1,\infty)$ we have $[L_{p}(S,w)]^{*} = L_{p'}(S,w')$
isometrically with respect to the pairing
\begin{equation}\label{DSOP:eq:subsec:prelim:weights;pairing}
L_{p}(S,w) \times L_{p'}(S,w') \longra \C,\, (f,g) \mapsto \int_{S}fg\,d\mu.
\end{equation}

For $p \in (1,\infty)$ we denote by $A_{p}=A_{p}(\R^{d})$ the class of all Muckenhoupt $A_{p}$-weights, which are all the locally integrable weights for which the $A_{p}$-characteristic $[w]_{A_{p}} \in [1,\infty]$ is finite.
We furthermore set $A_{\infty} := \bigcup_{p \in (1,\infty)}A_{p}$.

The relevant weights for this paper are the power weights of the form $w=\mathrm{dist}(\,\cdot\,,\partial\mathscr{O})^{\gamma}$, where $\mathscr{O}$ is a $C^{\infty}$-domain in $\R^{d}$ and where $\gamma \in (-1,\infty)$.
If $\mathscr{O} \subset \R^{d}$ is a Lipschitz domain and $\gamma \in \R$, $p \in (1,\infty)$, then (see \cite[Lemma~2.3]{Farwig&Sohr_Weighted_Lq-theory_Stokes_resolvent_exterior_domains} or \cite[Lemma~2.3]{Mitrea&Taylor_Poisson_problem_weighted_Sobolev_spaces_Lipschitz_domains})
\begin{equation}\label{DBVP:eq:sec:prelim:power_weight_Ap}
w_{\gamma}^{\mathscr{O}} := \mathrm{dist}(\,\cdot\,,\partial\mathscr{O})^{\gamma} \in A_{p} \quad \Longlra \quad \gamma \in (-1,p-1);
\end{equation}
in particular,
\begin{equation}\label{DBVP:eq:sec:prelim:power_weight_A_infty}
w_{\gamma}^{\mathscr{O}} = \mathrm{dist}(\,\cdot\,,\partial\mathscr{O})^{\gamma} \in A_{\infty} \quad \Longlra \quad \gamma \in (-1,\infty).
\end{equation}
For the important model problem case $\mathscr{O} = \R^{d}_{+}$ we simply write $w_{\gamma}:= w_{\gamma}^{\R^{d}_{+}} = \mathrm{dist}(\,\cdot\,,\partial\R^{d}_{+})^{\gamma}$.

Furthermore, in connection with the pairing \eqref{DSOP:eq:subsec:prelim:weights;pairing}, for $p \in (1,\infty)$ we have
\[
w \in A_{p} \quad \Longlra \quad w' \in A_{p'} \quad \Longlra \quad w,w' \in A_{\infty}.
\]

\subsection{UMD spaces and $L_{q}$-maximal regularity}\label{DSOP:subsec:prelim:UMD_max-reg}

The general references for this subsection are \cite{Hytonen&Neerven&Veraar&Weis2016_Analyis_in_Banach_Spaces_I,
Hytonen&Neerven&Veraar&Weis2016_Analyis_in_Banach_Spaces_II,Kunstmann&Weis2004_lecture_notes}.

The UMD property of Banach spaces is defined through the unconditionality of martingale differences, which is a primarily probabilistic notion.
A deep result due to Bourgain and Burkholder gives a pure analytic characterization in terms of the Hilbert transform: a Banach space $X$ has the UMD property if and only if it is of class $\mathcal{HT}$, i.e.\ the Hilbert transform $H$ has a bounded extension $H_{X}$ to $L^{p}(\R;X)$ for any/some $p \in (1,\infty)$.
A Banach space with the UMD property is called a UMD Banach space. Some facts:
\begin{itemize}
\item Every Hilbert space is a UMD space;
\item If $X$ is a UMD space, $(S, \Sigma, \mu)$ is $\sigma$-finite and $p\in (1, \infty)$, then $L^p(S;X)$ is a UMD space.
\item UMD spaces are reflexives.
\item Closed subspaces and quotients of UMD spaces are again UMD spaces.
\end{itemize}
In particular, weighted Besov and Triebel-Lizorkin spaces (see Section~\ref{DSOP:subsec:prelim:fs}) are UMD spaces in the reflexive range.

Let $A$ be a closed linear operator on a Banach space $X$. For $q \in (1,\infty)$ and $v \in A_{q}(\R)$ we say that $A$ enjoys the property of
\begin{itemize}
\item $L_{q}(v,\R)$-maximal regularity if $\frac{d}{dt}+A$ is invertible as an operator on $L_{q}(v,\R)$ with domain $W^{1}_{q}(\R,v;X) \cap L_{q}(\R,v;D(A))$.
\item $L_{q}(v,\R_{+})$-maximal regularity if $\frac{d}{dt}+A$ is invertible as an operator on $L_{q}(v,\R_{+})$ with domain ${_{0}}W^{1}_{q}(\R_{+},v;X) \cap L_{q}(\R_{+},v;D(A))$, where
    \[
    {_{0}}W^{1}_{q}(\R_{+},v;X) = \{ u \in W^{1}_{q}(\R_{+},v;X):u(0)=0\}.
    \]
\end{itemize}
In the specific case of the power weight $v=v_{\mu}$ with $q \in (-1,q-1)$, we speak of $L_{q,\mu}(\R)$-maximal regularity and $L_{q,\mu}(\R_{+})$-maximal regularity.

Note that $L_{q}(v,\R)$-maximal regularity and $L_{q}(v,\R_{+})$-maximal regularity can also be formulated in terms of evolution equations. For instance, $A$ enjoys the property of $L_{q}(v,\R_{+})$-maximal regularity if and only if, for each $f \in L_{q}(v,\R_{+})$, there exists a unique solution $u \in W^{1}_{q}(\R_{+},v;X) \cap L_{q}(\R_{+},v;D(A))$ of
\[
u'+Au = f, \quad u(0)=0.
\]

A way to approach $L_{q}$-maximal regularity is through the Dore-Venni operator sum theorem (see \cite{Dore&Venni1987,PrSo}). Using \cite[Proposition~2.7]{LV2018_Dir_Laplace} in combination with the easy connection between the $H^{\infty}$-calculus and bounded imaginary powers to do so, we obtain the following proposition. As the $H^{\infty}$-calculus will be used as a black box, we do not give the definition here but the reader refer to \cite[Section~2]{Lindemulder2018_elliptic_weighted_B&F}.
\begin{prop}
Let $X$ be a UMD space, $q \in (1,\infty)$ and $v \in A_{q}(\R)$. If $A$ be a closed linear operator on a Banach space $X$ with $0 \in \rho(A)$ that has a bounded $H^{\infty}$-calculus of angle $\omega_{H^{\infty}}(A) < \frac{\pi}{2}$, then $A$ enjoys the properties of $L_{q}(v,\R)$-maximal regularity and $L_{q}(v,\R_{+})$-maximal regularity.
\end{prop}

\subsection{Function spaces}\label{DSOP:subsec:prelim:fs}
The general references for this subsection are \cite{Amann2003_Vector-valued_distributions,Bui1982,LV2018_Dir_Laplace,
Meyries&Veraar2012_sharp_embedding,Meyries&Veraar2015_pointwise_multiplication}.

Let $X$ be a Banach space.
The space of $X$-valued tempered distributions $\mathcal{S}'(\R^{d};X)$ is defined as $\mathcal{S}'(\R^{d};X) := \mathcal{L}(\mathcal{S}(\R^{d}),X)$, the space of continuous linear operators from $\mathcal{S}(\R^{d})$ to $X$, equipped with the locally convex topology of bounded convergence.
Standard operators (derivative operators, Fourier transform, convolution, etc.) on $\mathcal{S}'(\R^{d};X)$ can be defined as in the scalar-case, see \cite[Section~III.4]{Amann1995_Lin_and_quasilinear_parabolic,Amann2003_Vector-valued_distributions}.

Let $\dom$ be an open subset of $\R^{d}$. For $p \in (1,\infty)$ and a weight $w$ on $\dom$ with $p$-dual weight $w'_{p} = w^{-\frac{1}{p-1}} \in L_{1,\mathrm{loc}}(\dom)$, there is the inclusion $L_{p}(\dom,w;X) \hookrightarrow \mathcal{D}'(\dom;X)$ (which can be seen through the pairing \eqref{DSOP:eq:subsec:prelim:weights;pairing}).
So, in this case, we can the Sobolev space $W^{k}_{p}(\dom,w;X)$ of order $k \in \N$ as the space of all $f \in \mathcal{D}'(\dom;X)$ with $D^{\alpha}f \in L_{p}(\dom,w;X)$ for each $|\alpha| \leq k$.
Equipped with the norm
\[
\norm{f}_{W^{k}_{p}(\dom,w;X)} := \sum_{|\alpha| \leq k}\norm{D^{\alpha}f}_{L_{p}(\dom,w;X)},
\]
$W^{k}_{p}(\dom,w;X)$ becomes a Banach space. An example of such a weight $w$ on a $C^{\infty}-$domain $\dom \subset \R^{d}$ for which the $p$-dual weight $w'_{p} = w^{-\frac{1}{p-1}} \in L_{1,\mathrm{loc}}(\dom)$ is the power weight $w_{\gamma}^{\partial\mathscr{O}} = \mathrm{dist}(\,\cdot\,,\partial\mathscr{O})^{\gamma}$ with $\gamma \in \R$.

Let $p \in (1,\infty)$ and $w \in A_{p}(\R^{d})$. Then $w'_{p} = w^{-\frac{1}{p-1}} \in A_{p'}$, so that $\mathcal{S}(\R^{d}) \stackrel{d}{\hookrightarrow} L_{p'}(\R^{d},w^{1-p'})$. Using the pairing, we find that $L^{p}(\R^{d},w;X) \hookrightarrow \mathcal{S}'(\R^{d};X)$ in the natural way.
For each $s \in \R$ we can thus define the Bessel potential space $H^{s}_{p}(\R^{d},w;X)$ as the space of all $f \in \mathcal{S}'(\R^{d};X)$
for which $\mathcal{J}_{s}f \in L_{p}(\R^{d},w;X)$, where $\mathcal{J}_{s} \in \mathcal{L}(\mathcal{S}'(\R^{d};X))$ is the Bessel potential operator given by
\[
\mathcal{J}_{s}f := \mathscr{F}^{-1}[(1+|\,\cdot\,|^{2})^{s/2}\hat{f}], \quad\quad f \in \mathcal{S}'(\R^{d};X).
\]
Equipped with the norm
\[
\norm{f}_{H^{s}_{p}(\R^{d},w;X)} := \norm{\mathcal{J}_{s}f}_{L^{p}(\R^{d},w;X)},
\]
$H^{s}_{p}(\R^{d},w;X)$ becomes a Banach space.
Note that $w'_{p} = w^{-\frac{1}{p-1}} \in L_{1,\mathrm{loc}}(\R^{d})$ as an $A_{p}$-weight and that  $H^{0}_{p}(\R^{d},w;X) = L_{p}(\R^{d},w;X) = W^{0}_{p}(\R^{d},w;X)$. If $X$ is a UMD space, then we have
\begin{equation}\label{DSOP:eq:prelim:H=W_UMD}
H^{n}_{p}(\R^{d},w;X) = W^{n}_{p}(\R^{d},w;X).
\end{equation}
In the reverse direction we have that if $H^{1}_{p}(\R;X) = W^{1}_{p}(\R;X)$, then $X$ is a UMD space (see \cite{Hytonen&Neerven&Veraar&Weis2016_Analyis_in_Banach_Spaces_I}).

For $0 < A < B < \infty$ we define $\Phi_{A,B}(\R^{d})$ as the set of all sequences $\varphi = (\varphi_{n})_{n \in \N} \subset \mathcal{S}(\R^{d};X)$ which can be constructed in the following way: given $\varphi_{0} \in \mathcal{S}(\R^{d})$ with
\[
0 \leq \hat{\varphi} \leq 1, \:\: \hat{\varphi}(\xi) = 1 \:\:\mbox{if $|\xi| \leq A$}, \:\:\hat{\varphi}(\xi) = 0 \:\:\mbox{if $|\xi| \geq B$},
\]
$(\varphi_{n})_{n \geq 1}$ is determined by
\[
\hat{\varphi}_{n} = \hat{\varphi}_{1}(2^{-n+1}\,\cdot\,) = \hat{\varphi}_{0}(2^{-n}\,\cdot\,) - \hat{\varphi}_{0}(2^{-n+1}\,\cdot\,), \quad\quad n \geq 1.
\]
Observe that
\begin{equation}\label{eq:Fourier_support_LP-seq}
\supp \hat{\varphi}_{0} \subset \{ \xi : |\xi| \leq B \} \quad \mbox{and} \quad \supp \hat{\varphi}_{n} \subset \{ \xi : 2^{n-1}A \leq |\xi| \leq 2^{n}B \}, \, n \geq 1.
\end{equation}
We furthermore put $\Phi(\R^{d}) := \bigcup_{0<A<B<\infty}\Phi_{A,B}(\R^{d})$.

Let $\varphi = (\varphi_{n})_{n \in \N} \in \Phi(\R^{d})$. We define the operators $\{S_{n}\}_{n \in \N} \subset  \mathcal{L}(\mathcal{S}'(\R^{d};X),\mathscr{O}_{M}(\R^{d};X))$ by
\[
S_{n}f := \varphi_{n}*f = \mathscr{F}^{-1}[\hat{\varphi}_{n}\hat{f}], \quad\quad f \in \mathcal{S}'(\R^{d};X),
\]
where $\mathscr{O}_{M}(\R^{d};X)$ stands for the space of all $X$-valued slowly increasing smooth functions on $\R^{d}$.
Given $s \in \R$, $p \in [1,\infty)$, $q \in [1,\infty]$ and $w \in A_{\infty}(\R^{d})$, the Besov space $B^{s}_{p,q}(\R^{d},w;X)$ is defined as the space of all $f \in \mathcal{S}'(\R^{d};X)$ for which
\[
\norm{f}_{B^{s}_{p,q}(\R^{d},w;X)} := \norm{ (2^{sn}S_{n}f)_{n \in \N} }_{\ell_{p}(\N)[L_{q}(\R^{d},w)(X)} < \infty.
\]
and Triebel-Lizorkin space $F^{s}_{p,q}(\R^{d},w;X)$ is defined as the space of all $f \in \mathcal{S}'(\R^{d};X)$ for which
\[
\norm{f}_{F^{s}_{p,q}(\R^{d},w;X)} := \norm{ (2^{sn}S_{n}f)_{n \in \N} }_{L_{p}(\R^{d},w)[\ell_{q}(\N)](X)} < \infty.
\]
Each choice of $\varphi \in \Phi(\R^{d})$ leads to an equivalent extended Banach norm on $\mathcal{S}'(\R^{d};X)$ for each of $B^{s}_{p,q}(\R^{d},w;X)$ and $F^{s}_{p,q}(\R^{d},w;X)$.

Let $\E \hookrightarrow \mathcal{D}'(U)$ be a Banach space of distributions on an open subset $U \subset \R^{d}$.
Given an open subset $V \subset U$,
\[
\E(V) := \{ f \in \mathcal{D}'(V) : \exists g \in \E, g_{|V} = f  \}
\]
equipped with the norm
\[
\norm{f}_{\E(V)} := \inf\{ \norm{g}_{\E} : g \in \E, g_{|V} = f \}
\]
is a Banach space with $\E(V) \hookrightarrow \mathcal{D}'(V)$.
We write
\[
F^{s}_{p,q}(U,w) := [F_{p,q}(\R^{d},w)](U),\quad
H^{s}_{p}(U,w) := [H^{s}_{p}(\R^{d},w)](U).
\]
If $\dom$ be a $C^{\infty}$-domain in $\R^{d}$, we write
\[
F^{s}_{p,q,\gamma}(\mathscr{O}) := F^{s}_{p,q}(\mathscr{O},w^{\partial\mathscr{O}}_{\gamma}), \quad
H^{s}_{p,\gamma}(\mathscr{O}) := H^{s}_{p}(\mathscr{O},w^{\partial\mathscr{O}}_{\gamma}),
W^{k}_{p,\gamma}(\mathscr{O}) := W^{k}_{p}(\mathscr{O},w^{\partial\mathscr{O}}_{\gamma}).
\]

The $H$-spaces are related to the $F$-spaces as follows. In the scalar-valued case $X=\C$, we have
\begin{equation}\label{eq:identity_Bessel-Potential_Triebel-Lizorkin;scalar-valued}
H^{s}_{p}(\R^{d},w) = F^{s}_{p,2}(\R^{d},w), \quad\quad p \in (1,\infty), w \in A_{p}.
\end{equation}
In the vector-valued case, this identity is valid if and only if $X$ is isomorphic to a Hilbert space.
For general Banach spaces $X$ we still have (see \cite[Proposition~3.12]{Meyries&Veraar2012_sharp_embedding})
\begin{equation}\label{eq:relation_Bessel-Potential_Triebel-Lizorkin;Ap}
F^{s}_{p,1}(\R^{d},w;X) \hookrightarrow H^{s}_{p}(\R^{d},w;X) \hookrightarrow F^{s}_{p,\infty}(\R^{d},w;X),
\quad\quad p \in (1,\infty),w \in A_{p}(\R^{d}),
\end{equation}
and (see \cite[(7.1)]{LV2018_Dir_Laplace})
\begin{equation}\label{DSOP:eq:prelim:emb_F_into_Lp}
F^{k}_{p,1,\gamma}(\dom;X)
\hookrightarrow W^{k}_{p,\gamma}(\dom;X), \qquad k \in \N, p \in (1,\infty) \gamma \in (-1,\infty),
\end{equation}
where $\dom \subset \R^{d}$ is a $C^{\infty}$-domain with compact boundary.

For UMD spaces $X$ there is a suitable randomized substitute for \eqref{eq:identity_Bessel-Potential_Triebel-Lizorkin;scalar-valued} (see \cite[Proposition~3.2]{Meyries&Veraar2015_pointwise_multiplication}), which in the UMD Banach function space case again reduces to a square function description.
That Littlewood-Paley decomposition can be formulated using the generalized Triebel-Lizorkin spaces $\F^{s}_{p,q}(\R^{d},w;E)$, which are defined as follows.

The following special case of the sharp Sobolev embedding from \cite{Meyries&Veraar2012_sharp_embedding} is a crucial ingredient for this paper:
\begin{equation}\label{DSOP:eq:prelim:Sob_embd}
F^{s_{0}}_{p,r_{0},\gamma_{0}}(\dom) \hookrightarrow F^{s_{1}}_{p,r_{1},\gamma_{1}}(\dom), \qquad \gamma_{1} > \gamma_{0}, s_{0} = s_{1}+\frac{\gamma_{0}-\gamma_{1}}{p}, p \in (1,\infty),
\end{equation}
where $\dom$ is a $C^{\infty}$-domain in $\R^{d}$ with a compact boundary.

From \cite{Sickel&Skrzypczak&Vybiral2014} and a retraction-coretraction argument using for instance Rychkov's extension operator \cite{Rychkov_Restr&extensions_B&F-spaces_Lipschitz_domains} (see \cite[Theorem~2.1]{Lindemulder2018_elliptic_weighted_B&F}) we obtain the following complex interpolation result: if $\dom$ is a Lipschit domain in $\R^{d}$ with a compact boundary, $p_{j} \in [1,\infty)$, $q_{j} \in [1,\infty]$, $w_{j} \in A_{\infty}$, $s_{j} \in \R$ for $j \in \{0,1\}$, $\theta \in (0,1)$, $\frac{1}{p}=\frac{1-\theta}{p_{0}}+\frac{\theta}{p_{1}}$, $\frac{1}{q}=\frac{1-\theta}{q_{0}}+\frac{\theta}{q_{1}}$, $w=w_{0}^{(1-\theta)p/p_{0}}w_{1}^{\theta p/p_{1}}$ and $s=(1-\theta)s_{0}+\theta s_{1}$, then
\begin{equation}\label{DSOP:eq:prelim:complex_int_F}
F^{s}_{p,q}(\dom,w) = [F^{s_{0}}_{p_{0},q_{0}}(\dom,w_{0}),F^{s_{1}}_{p_{1},q_{1}}(\dom,w_{1})]_{\theta}.
\end{equation}

Suppose $\R^{d}=\R^{k} \times \R^{m}$. For $a,b \in (0,\infty)$ we define $\Phi^{(a,b)}(\R^{k} \times \R^{m})$ the $(a,b)$-anisotropic analogue of $\Phi(\R^{d})$ by simply replacing the standard isotropic scaling $\delta_{\lambda}(x)=\lambda x$ in the definition by the $(a,b)$-anisotropic scaling
\[
\delta^{(a,b),k \times m}_{\lambda}(x,y) = (\lambda^{a}x,\lambda^{b}y), \qquad (x,y) \in \R^{k} \times \R^{m},
\]
see \cite{Lindemulder2017_max-reg} for more details.
Let $\varphi = (\varphi_{n})_{n \in \N} \in \Phi^{(a,b)}(\R^{k} \times \R^{m})$ and let $\{S_{n}\}_{n \in \N} \subset  \mathcal{L}(\mathcal{S}'(\R^{d};X),\mathscr{O}_{M}(\R^{d};X))$ be the associated sequence of convolution operators.
For $p,q \in [1,\infty)$, $r \in [1,\infty]$, $v \in A_{\infty}(\R^{m})$, $w \in A_{\infty}(\R^{n})$, $a,b \in (0,\infty)$ and $s \in \R$ we define the anisotropic mixed-norm Triebel-Lizorkin space $F^{s,(a,b)}_{(p,q),r}(\R^{k} \times \R^{m},(w,v))$ as the space of all $f \in \mathcal{S}'(\R^{d})$ with
\[
\norm{f}_{F^{s,(a,b)}_{(p,q),r}(\R^{k} \times R^{m},(w,v))} := \norm{ (2^{sn}S_{n}f)_{n \in \N} }_{L_{(p,q)}(\R^{k} \times \R^{m},(w,v))[\ell_{r}(\N)]} < \infty,
\]
where $L_{(p,q)}(\R^{k} \times \R^{m},(w,v)) = L_{q}(\R^{m},v)[L_{p}(\R^{k},w)]$.
Each choice of $\varphi \in \Phi^{(a,b)}(\R^{k} \times \R^{m})$ leads to an equivalent extended Banach norm on $\mathcal{S}'(\R^{d};X)$.
We put $F^{s,(a,b)}_{(p,q),r}(\R^{k}_{+} \times \R^{m},(w,v))= \left[F^{s,(a,b)}_{(p,q),r}(\R^{k} \times \R^{m},(w,v);X)\right](\R^{k}_{+} \times \R^{m})$.

\section{A Trace Theorem}\label{DSOP:sec:trace}

\begin{thm}\label{DSOP:thm:trace_MR}
Let $\mathscr{O}$ be either $\R^{d}_{+}$ or a $C^{\infty}$-domain in $\R^{d}_{+}$ with a compact boundary $\partial\mathscr{O}$.
Let $q,p,r \in (1,\infty)$, $v \in A_{q}(\R)$, $\gamma \in (-1,\infty)$, $s \in \R$ and $\rho \in (0,\infty)$. Put $\delta := \frac{s}{\rho}+1-\frac{1+\gamma}{\rho p}$.
If $s < \frac{1+\gamma}{p} < s+\rho$, then taking the trace with respect to $\partial\dom$ gives a retraction
\begin{align*}
\mathrm{tr}_{\partial\dom}: W^{1}_{q}(\R,v;F^{s}_{p,r,\gamma}(\dom) \cap & L_{q}(\R,v;F^{s+\rho}_{p,r,\gamma}(\dom) \\
& \longra F^{\delta}_{q,p}(\R,v;L_{p}(\partial\dom)) \cap L_{q}(\R,v;B^{\rho\delta}_{p,p}(\partial\dom)).
\end{align*}
\end{thm}

Before we give the proof of Theorem~\ref{DSOP:thm:trace_MR}, we start with some preparatory lemma's.
\begin{lemma}\label{DSOP:lemma:embd_Lp_into_neg_smoothness}
Let $p,r \in (1,\infty)$, $\gamma \in (-1,\infty)$ and $\delta \in (0,\infty)$ be such that $\gamma-\delta p \in (-1,p-1)$.
Then
\[
L_{p}(\R^{d}_{+},w_{\gamma}) \hookrightarrow F^{-\delta}_{p,r}(\R^{d}_{+},w_{\gamma-\delta p}).
\]
\end{lemma}
\begin{proof}
Recall from \cite[Section~5.2]{Lindemulder2018_elliptic_weighted_B&F} that $F^{-\delta}_{p,r}(\R^{d}_{+},w_{\gamma-\delta p})$ is a reflexive Banach space with
\[
\mathcal{D}(\R^{d}_{+}) \stackrel{d}{\hookrightarrow} F^{-\delta}_{p,r}(\R^{d}_{+},w_{\gamma-\delta p}) \hookrightarrow \mathcal{D}'(\R^{d}_{+}),
\]
so that
\[
\mathcal{D}(\R^{d}_{+}) \stackrel{d}{\hookrightarrow} [F^{-\delta}_{p,r}(\R^{d}_{+},w_{\gamma-\delta p})]^{*} \hookrightarrow \mathcal{D}'(\R^{d}_{+}),
\]
under the natural identifications. By \cite[Lemma~4.19]{Lindemulder2018_elliptic_weighted_B&F}, \begin{equation}\label{DSOP:eq:lemma:embd_Lp_into_neg_smoothness;1}
[F^{-\delta}_{p,r}(\R^{d}_{+},w_{\gamma-\delta p})]^{*} = \mathring{F}^{\delta}_{p',r'}(\R^{d}_{+},w_{\gamma'+\delta p'}),
\end{equation}
where $\mathring{F}^{\delta}_{p',r'}(\R^{d}_{+},w_{\gamma'+\delta p'})$ denotes the closure of $\mathcal{D}(\R^{d}_{+})$ in $F^{\delta}_{p',r'}(\R^{d}_{+},w_{\gamma'+\delta p'})$.
Now pick $\ell \in [0,\delta) \cap \N$ such that $\gamma'+\ell p' \in (-1,p'-1)$. Then
\begin{equation}\label{DSOP:eq:lemma:embd_Lp_into_neg_smoothness;2}
\mathring{F}^{\delta}_{p',r'}(\R^{d},w_{\gamma'+\delta p'}) \hookrightarrow \mathring{F}^{\ell}_{p',1}(\R^{d},w_{\gamma'+\ell p'})
\end{equation}
as a consequence of the Sobolev embedding \eqref{DSOP:eq:prelim:Sob_embd}.
By a combination of the elementary embedding \eqref{DSOP:eq:prelim:emb_F_into_Lp} and the Sobolev embedding \cite[Corollary~3.4]{LV2018_Dir_Laplace},
\begin{equation}\label{DSOP:eq:lemma:embd_Lp_into_neg_smoothness;4}
\mathring{F}^{\ell}_{p',1}(\R^{d}_{+},w_{\gamma'+\ell p'})  \hookrightarrow L_{p'}(\R^{d}_{+},w_{\gamma'}).
\end{equation}
Putting together \eqref{DSOP:eq:lemma:embd_Lp_into_neg_smoothness;1}, \eqref{DSOP:eq:lemma:embd_Lp_into_neg_smoothness;2}, \eqref{DSOP:eq:lemma:embd_Lp_into_neg_smoothness;4}, we obtain
\[
[F^{-\delta}_{p,r}(\R^{d}_{+},w_{\gamma-\delta p})]^{*} \hookrightarrow L_{p'}(\R^{d}_{+},w_{\gamma'}).
\]
Moreover, as $\mathcal{D}(\R^{d}_{+})$ is dense in each of these two spaces, the inclusion is dense.
Dualizing finally gives the desired inclusion.
\end{proof}

\begin{lemma}\label{DSOP:lemma:embd_anisotropic_MR}
Let $q,p,r \in (1,\infty)$, $v \in A_{q}(\R)$, $\gamma \in (-1,\infty)$, $s \in \R$ and $\rho \in (0,\infty)$.
Let $\tilde{s} \in (0,\infty)$ with $\tilde{s} \geq s$, $\tilde{\gamma} := \gamma + (\tilde{s}-s)p$ and $\sigma := \frac{\tilde{s}}{p}+1$.
Let $\delta \in (0,\infty)$ be such that $\tilde{\gamma}-\delta p \in (-1,p-1)$ and put $\eta:= \frac{1}{\sigma-1}\delta$. Then
\begin{align}
& F^{\sigma,(\frac{1}{\rho},1)}_{(p,q),1}(\R^{d}_{+} \times \R,(w_{\tilde{\gamma}},v)) \cap F^{\sigma+\frac{\eta}{\rho},(\frac{1}{\rho},1)}_{(p,q),1}(\R^{d}_{+} \times \R,(w_{\tilde{\gamma}+\eta p},v)) \nonumber \\
&\qquad\qquad\qquad \hookrightarrow \quad W^{1}_{q}(\R,v;F^{s}_{p,r}(\R^{d}_{+},w_{\gamma})) \cap L_{q}(\R,v;F^{s+\rho}_{p,r}(\R^{d}_{+},w_{\gamma})). \label{DSOP:eq:lemma:embd_anisotropic_MR}
\end{align}
\end{lemma}
\begin{proof}
By Sobolev embedding \eqref{DSOP:eq:prelim:Sob_embd} we have
\begin{align*}
& W^{1}_{q}(\R,v;F^{\tilde{s}}_{p,r}(\R^{d}_{+},w_{\tilde{\gamma}})) \cap L_{q}(\R,v;F^{\tilde{s}+\rho}_{p,r}(\R^{d}_{+},w_{\tilde{\gamma}})) \\
&\qquad\qquad\qquad \hookrightarrow \quad W^{1}_{q}(\R,v;F^{s}_{p,r}(\R^{d}_{+},w_{\gamma})) \cap L_{q}(\R,v;F^{s+\rho}_{p,r}(\R^{d}_{+},w_{\gamma})).
\end{align*}
Therefore, we may without loss of generality assume that $s = \tilde{s}$ and $\tilde{\gamma} = \gamma$.

By a combination of \cite[Examples~5.9]{Lindemulder_Intersection} and \cite[Examples~5.10]{Lindemulder_Intersection} (using the fact that $\C$ is a UMD Banach space and that $\C$ has type 1), we find that
\begin{equation}\label{DSOP:eq:lemma:embd_anisotropic_MR;2}
F^{\sigma+\frac{\eta}{\rho},(\frac{1}{\rho},1)}_{(p,q),1}(\R^{d}_{+} \times \R,(w_{\gamma+\eta p},v))  \hookrightarrow L_{q}(\R,v;F^{\rho\sigma+\eta}_{p,1}(\R^{d}_{+},w_{\gamma+\eta p}))
\end{equation}
and 
\begin{equation}\label{DSOP:eq:lemma:embd_anisotropic_MR;1}
F^{\sigma,(\frac{1}{\rho},1)}_{(p,q),1}(\R^{d}_{+} \times \R,(w_{\gamma},v)) \hookrightarrow H^{\sigma}_{q}(\R,v;L_{p}(\R^{d}_{+},w_{\gamma})) \cap L_{q}(\R,v;F^{\rho\sigma}_{p,1}(\R^{d}_{+},w_{\gamma})).
\end{equation}

Combining \eqref{DSOP:eq:lemma:embd_anisotropic_MR;2}, \eqref{DSOP:eq:lemma:embd_anisotropic_MR;1} and Lemma~\ref{DSOP:lemma:embd_Lp_into_neg_smoothness}, we find
\begin{align}
& F^{\sigma,(\frac{1}{\rho},1)}_{(p,q),1}(\R^{d}_{+} \times \R,(w_{\gamma},v)) \cap F^{\sigma+\frac{\eta}{\rho},(\frac{1}{\rho},1)}_{(p,q),1}(\R^{d}_{+} \times \R,(w_{\gamma+\eta p},v)) \nonumber \\
&\qquad\qquad\qquad \hookrightarrow \quad H^{\sigma}_{q}(\R,v;L_{p}(\R^{d}_{+},w_{\gamma})) \cap L_{q}(\R,v;F^{\rho\sigma+\eta}_{p,1}(\R^{d}_{+},w_{\gamma+\eta p})) \nonumber \\
&\qquad\qquad\qquad \hookrightarrow \quad H^{\sigma}_{q}(\R,v;F^{-\delta}_{p,r}(\R^{d}_{+},w_{\gamma-\delta p})) \cap L_{q}(\R,v;F^{\rho\sigma+\eta}_{p,r}(\R^{d}_{+},w_{\gamma+\eta p})). \label{DSOP:eq:lemma:embd_anisotropic_MR;3}
\end{align}

From a combination of \cite[Theorem~3.18]{LV2018_Dir_Laplace},
\eqref{DSOP:eq:prelim:complex_int_F} and \eqref{DSOP:eq:prelim:H=W_UMD} it follows that
\begin{align}
& H^{\sigma}_{q}(\R,v;F^{-\delta}_{p,r}(\R^{d}_{+},w_{\gamma-\delta p})) \cap L_{q}(\R,v;F^{\rho\sigma+\eta}_{p,r}(\R^{d}_{+},w_{\gamma+\eta p})) \nonumber \\
&\qquad\qquad\qquad \hookrightarrow \quad
[L_{q}(\R,v;F^{\rho\sigma+\eta}_{p,r}(\R^{d}_{+},w_{\gamma+\eta p})),H^{\sigma}_{q}(\R,v;F^{-\delta}_{p,r}(\R^{d}_{+},w_{\gamma-\delta p}))]_{1/\sigma} \nonumber \\
&\qquad\qquad\qquad\:\: = \quad W^{1}_{q}(\R,v;F^{s}_{p,r}(\R^{d}_{+},w_{\gamma})), \label{DSOP:eq:lemma:embd_anisotropic_MR;4}
\end{align}
where we used that $\eta(1-\frac{1}{\sigma})+(-\delta)\frac{1}{\sigma}=0$ and $\rho\sigma(1-\frac{1}{\sigma}) = s$ and the fact that $F^{\rho\sigma+\eta}_{p,r}(\R^{d}_{+},w_{\gamma+\eta p}))$ and $F^{-\delta}_{p,r}(\R^{d}_{+},w_{\gamma-\delta p})$ are UMD spaces.

Finally, combining \eqref{DSOP:eq:lemma:embd_anisotropic_MR;1}, \eqref{DSOP:eq:lemma:embd_anisotropic_MR;3} and \eqref{DSOP:eq:lemma:embd_anisotropic_MR;4} and using that $\rho\sigma=s+\rho$, we arrive at the desired inclusion.
\end{proof}

\begin{lemma}\label{DSOP:lemma:embd_MR_anisotropic}
Let $q,p \in (1,\infty)$, $v \in A_{q}(\R)$, $\gamma \in (-1,\infty)$, $s \in \R$ and $\rho \in (0,\infty)$. If $\theta \in [0,1]$ is such that $s+\theta\rho \in (0,\infty) \cap (\frac{1+\gamma}{p}-1,\frac{1+\gamma}{p})$, then
\begin{align}
&W^{1}_{q}(\R,v;F^{s}_{p,\infty}(\R^{d}_{+},w_{\gamma})) \cap L_{q}(\R,v;F^{s+\rho}_{p,\infty}(\R^{d}_{+},w_{\gamma})) \nonumber \\
&\qquad\qquad \hookrightarrow \quad H^{1-\theta}_{q}(\R,v;L_{p}(\R^{d}_{+},w_{\gamma-(s+\theta\rho)p})) \cap L_{q}(\R,v;H^{(1-\theta)\rho}_{p}(\R^{d}_{+},w_{\gamma-(s+\theta\rho)p})). \label{DSOP:eq:lemma:embd_MR_anisotropic}
\end{align}
\end{lemma}
Note that $s+\theta\rho \in (\frac{1+\gamma}{p}-1,\frac{1+\gamma}{p})$ is equivalent to $\gamma-(s+\theta\rho)p \in (-1,p-1)$, which is in turn equivalent to $w_{\gamma-(s+\theta\rho)p} \in A_{p}$.

\begin{proof}
By Sobolev embedding \eqref{DSOP:eq:prelim:Sob_embd} we have
\begin{align*}
&W^{1}_{q}(\R,v;F^{s}_{p,\infty}(\R^{d}_{+},w_{\gamma})) \cap L_{q}(\R,v;F^{s+\rho}_{p,\infty}(\R^{d}_{+},w_{\gamma})) \nonumber \\
& \qquad\qquad \hookrightarrow W^{1}_{q}(\R,v;H^{-\theta\rho}_{p}(\R^{n}_{+},w_{\gamma-(s+\theta\rho)p}))
\cap L_{q}(\R,v;H^{(1-\theta)\rho}_{p}(\R^{n}_{+},w_{\gamma-(s+\theta\rho)p})). \label{DSOP:eq:lemma:embd_MR_anisotropic;1}
\end{align*}
Since
\begin{align*}
&W^{1}_{q}(\R,v;H^{-\theta\rho}_{p}(\R^{n}_{+},w_{\gamma-(s+\theta\rho)p}))
\cap L_{q}(\R,v;H^{(1-\theta)\rho}_{p}(\R^{n}_{+},w_{\gamma-(s+\theta\rho)p})). \\
& \qquad\qquad \hookrightarrow [W^{1}_{q}(\R,v;H^{-\theta\rho}_{p}(\R^{n}_{+},w_{\gamma-(s+\theta\rho)p})),
L_{q}(\R,v;H^{(1-\theta)\rho}_{p}(\R^{n}_{+},w_{\gamma-(s+\theta\rho)p}))]_{\theta} \\
& \qquad\qquad = H^{1-\theta}_{q}(\R,v;L_{p}(\R^{d}_{+},w_{\gamma-(s+\theta\rho)p}))
\end{align*}
as a consequence of \cite[Theorem~3.18]{LV2018_Dir_Laplace}, \cite[Proposition~5.6]{Lindemulder&Meyries&Veraar2017} and \cite[Proposition~5.5]{Lindemulder&Meyries&Veraar2017} (the half-space version of \eqref{DSOP:eq:prelim:H=W_UMD}), the desired inclusion follows.
\end{proof}

\begin{proof}[Proof of Theorem~\ref{DSOP:thm:trace_MR}]
By localization and rectification it suffices to treat the case $\dom = \R^{d}_{+}$.
It will be convenient to write
\begin{align*}
\M &:=
W^{1}_{q}(\R,v;F^{s}_{p,r}(\R^{d}_{+},w_{\gamma})) \cap  L_{q}(\R,v;F^{s+\rho}_{p,r}(\R^{d}_{+},w_{\gamma})), \\
\B &:= F^{\delta}_{q,p}(\R,v;L_{p}(\R^{d-1})) \cap L_{q}(\R,v;B^{\rho\delta}_{p,p}(\R^{d-1})).
\end{align*}

As $s < \frac{1+\gamma}{p} < s+\rho$, we have $s < \frac{1+\gamma}{p}$ and $s+\rho > \max\{0,\frac{1+\gamma}{p}-1\}$ and there exists $\theta \in [0,1]$ such that $s+\theta\rho \in (0,\infty) \cap (\frac{1+\gamma}{p}-1,\frac{1+\gamma}{p})$. So the inclusion \eqref{DSOP:eq:lemma:embd_MR_anisotropic} from Lemma~\ref{DSOP:lemma:embd_MR_anisotropic} is valid. Combing this inclusion with an elementary inclusion and the trace theory from \cite[Theorem~2.1 $\&$ Corollary~4.9]{Lindemulder2017_max-reg} (also see \cite[Theorem~4.4]{Lindemulder2017_max-reg}) and observing that $(1-\theta)-\frac{1+(\gamma-(s+\theta\rho)p)}{\rho p} =\delta > 0$, we find that $\mathrm{tr}_{\partial\dom}$ is a bounded operator
\begin{align*}
\mathrm{tr}_{\partial\R^{d}_{+}}:\M \longra F^{\delta}_{q,p}(\R,v;L_{p}(\R^{d-1})) \cap L_{q}(\R,v;B^{\rho\delta}_{p,p}(\R^{d-1})).
\end{align*}

Let $s_{+}$, $\gamma_{+}$, $\sigma$, $\delta$ and $\eta$ be as in Lemma~\ref{DSOP:eq:lemma:embd_anisotropic_MR}. Observing that $(\sigma+\frac{\eta}{\rho})-\frac{1+(\gamma_{+}+\eta p)}{\rho p} = \sigma - \frac{1+\gamma_{+}}{\rho p} = \delta$, \cite[Theorem~2.1 $\&$ Corollary~4.9]{Lindemulder2017_max-reg} gives a bounded linear operator
\begin{align*}
E:\B \longra F^{\sigma,(\frac{1}{\rho},1)}_{(p,q),1}(\R^{d}_{+} \times \R,(w_{\gamma_{+}},v)) \cap F^{\sigma+\frac{\eta}{\rho},(\frac{1}{\rho},1)}_{(p,q),1}(\R^{d}_{+} \times \R,(w_{\gamma_{+}+\eta p},v))
\stackrel{\eqref{DSOP:eq:lemma:embd_anisotropic_MR}}{\hookrightarrow} \M
\end{align*}
that acts a coretraction for $\mathrm{tr}_{\partial\R^{d}_{+}}$.
\end{proof}

\section{$L_{q,\mu}$-$F^{s}_{p,r,\gamma}$-maximal Regularity}\label{DSOP:sec:max-reg}

\subsection{Assumptions on the Differential Operator}\label{DBVP:sec:assump:operator}

Let $\mathscr{O}$ be either $\R^{d}_{+}$ or a $C^{\infty}$-domain in $\R^{d}$ with a compact boundary $\partial \mathscr{O}$.
In this section we consider second order differential operators $\mathcal{A}(\,\cdot\,,D) = \sum_{|\alpha| \leq 2}D^{\alpha}$ on $\mathscr{O}$ with $a_{\alpha} \in BUC(\mathscr{O})$ for $|\alpha|=2$ and $a_{\alpha} \in \mathcal{D}'(\mathscr{O})$ for $|\alpha| \leq 1$. We will impose an ellipticity condition and certain smoothness conditions on the coefficients, which we describe below.

The ellipticity condition reads as follows:
\begin{itemize}
\item[$(\mathrm{E})$] There exists a constant $\kappa > 0$, called the \emph{constant of ellipticity}, such that for all $x,\xi \in \R^{d}$,
    \[
    \kappa^{-1}|\xi|^{2} \geq \sum_{|\alpha| = 2}a_{\alpha}(x)\xi^{\alpha} \geq \kappa|\xi|^{2}.
    \]
\end{itemize}

In order to formulate the smoothness condition on the coefficient, let $p,r \in (1,\infty)$, $\gamma \in (-1,\infty)$ and $s \in \R$.
\begin{itemize}
\item[$(\mathrm{S})^{s}_{p,r,\gamma}$] There exist $\kappa > \sigma_{s,p,\gamma}$ and $\vartheta \in (0,1)$ such that:
    \begin{enumerate}[(i)]
    \item\label{DSOE:it:S_top_order_coeff} $a_{\alpha} \in B^{\kappa}_{\infty,1}(\mathscr{O})$ for each $|\alpha|=2$;
    \item\label{DSOE:it:S_lower_order_coeff} $a_{\alpha}D^{\alpha} \in \mathcal{B}(F^{s+2\vartheta}_{p,q,\gamma}(\mathscr{O}),F^{s}_{p,q,\gamma}(\mathscr{O}))$ for each $|\alpha| \leq 1$.
    \end{enumerate}
    Moreover, there are the limits at infinity $a_{\alpha}(\infty) := \lim_{|x| \to \infty}a_{\alpha}(x)$ for each $|\alpha|=2$.
\end{itemize}

\begin{ex}\label{DSOE:ex:smoothness_cond}
Two explicit examples for the condition \eqref{DSOE:it:S_lower_order_coeff} in $(\mathrm{S})^{s}_{p,r,\gamma}$ on the lower order part are the following:
\begin{enumerate}[(a)]
\item There exists $\kappa_{i} > \sigma_{s+2-i,s,p,\gamma}$, $i=0,1$, such $a_{\alpha} \in B^{\kappa_{|\alpha|}}_{\infty,1}(\mathscr{O})$ for each $|\alpha| \leq 1$.
\item $a_{\alpha} \in L_{\infty}(\mathscr{O})$ for each $|\alpha| \leq 1$ in case $s=0$, $q=2$ and $\gamma \in (-1,p-1)$ (so that $F^{s}_{p,q,\gamma}(\mathscr{O}) = F^{0}_{p,2,\gamma}(\mathscr{O}) = L_{p,\gamma}(\mathscr{O})$).
\end{enumerate}
\end{ex}

\subsection{The $L_{q,\mu}$-$F^{s}_{p,r,\gamma}$-maximal Regularity Problem}

Let us first introduce some notation.
Let $\dom$ be either $\R^{d}_{+}$ or a $C^{\infty}$-domain in $\R^{d}$ with a compact boundary $\partial\dom$. Let $q,p,r\in (1, \infty)$, $v\in A_q(\R)$ and $\gamma\in (-1, \infty)$. For an interval $J \subset \R$ we set $\mathbb{D}^{q,p,r}_{v,\gamma,s}(J) := L_q(J, v;F^{s}_{p,r,\gamma}(\dom))$,
\[
\M^{q,p,r}_{v,\gamma,s}(J) := W^{1}_{q}(J,v;F^{s}_{p,r,\gamma}(\dom))\cap L_{q}(J,v;F^{s+2}_{p,r,\gamma}(\dom))
\]
and
\[
\B^{q,p,r}_{v,\gamma,s}(J) := F^{1+\frac{s}{2}-\frac{1}{2}\frac{1+\gamma}{p}}_{q,p}(J,v;L_{p}(\partial\dom)) \cap
L_{q}(J,v;B^{2+s-\frac{1+\gamma}{p}}_{p,p}(\partial\dom)).
\]
For the power weight $v=v_{\mu}$, with $\mu \in (-1,q-1)$, we simply replace $v$ by $\mu$ in the subscripts: $\mathbb{D}^{q,p,r}_{\mu,\gamma,s}(J) := \mathbb{D}^{q,p,r}_{v_{\mu},\gamma,s}(J)$, $\M^{q,p,r}_{\mu,\gamma,s}(J) := \M^{q,p,r}_{v_{\mu},\gamma,s}(J)$ and $\B^{q,p,r}_{\mu,\gamma,s}(J)  := \B^{q,p,r}_{v_{\mu},\gamma,s}(J)$.

\begin{thm}\label{DSOP:thm:main;real_line}
Let $\dom$ be either $\R^{d}_{+}$ or a $C^{\infty}$-domain in $\R^{d}$ with a compact boundary $\partial\dom$.
Let $q,p,r \in (1, \infty)$, $v\in A_q(\R)$, $\gamma\in (-1,\infty)$ and $s \in (\frac{1+\gamma}{p}-2,\frac{1+\gamma}{p})$.
Suppose that $\mathcal{A}(\,\cdot\,,D) = \sum_{|\alpha| \leq 2}a_{\alpha}D^{\alpha}$ satisfies the smoothness condition $(\mathrm{S})^{s}_{p,r,\gamma}$ and the ellipticity condition $(\mathrm{E})$.
Then there exists $\lambda_{0} \in \R$ such that for all $\lambda \geq \lambda_{0}$
\begin{equation}\label{DSOP:eq:thm:main;real_line;isom}
\M^{q,p,r}_{v,\gamma,s}(\R) \longra \mathbb{D}^{q,p,r}_{v,\gamma,s}(\R) \oplus \B^{q,p,r}_{v,\gamma,s}(\R),\,u \mapsto (u'+(\lambda+\mathcal{A}(\,\cdot\,,D))u,\tr_{\partial\dom}u)
\end{equation}
defines an isomorphism of Banach spaces; in particular, for all $\lambda \geq \lambda_{0}$, $f \in \mathbb{D}^{q,p,r}_{v,\gamma,s}(\R)$ and $g \in \B^{q,p,r}_{v,\gamma,s}(\R)$, there exists a unique solution $u \in \M^{q,p,r}_{v,\gamma,s}(\R)$ of the parabolic boundary value problem
\begin{align*}
\left\{\begin{array}{rl}
u' + (\lambda+\mathcal{A}(\,\cdot\,,D)) u &= f, \\
\mathrm{tr}_{\partial\dom}u &=g. \\
\end{array}\right.
\end{align*}
\end{thm}

\begin{remark}\label{DSOP:rmk:thm:main;real_line;lambda0}
In case that $\dom$ is bounded,
\cite[Corollary~9.7 and Example~9.8]{Lindemulder2018_elliptic_weighted_B&F} provides some information on $\lambda_{0}$ in Theorem~\ref{DSOP:thm:main;real_line}.
In particular, $\lambda_{0} \in \R$ can be chosen independently of $q,p,r$, $v$, $\gamma$ and $s$ (for which the smoothness condition $(\mathrm{S})^{s}_{p,r,\gamma}$ is satisfied).
Moreover, if $\dom$ is bounded and the operator $\mathcal{A}(\,\cdot\,,D)$ is in divergence form $\mathcal{A}(\,\cdot\,,D) = \sum_{i,j=1}^{d}D_{j}(a^{i,j}D_{i})$ with $a^{i,j} \in BC^{\infty}(\mathscr{O})$, we can take $\lambda_{0}=0$ (or even a certain $\lambda_{0} < 0$).
\end{remark}

\begin{proof}
The required boundedness of the mapping $u \mapsto (u'+(\lambda-\Delta)u,\tr_{\partial\dom}u)$ follows from the smoothness condition $(\mathrm{S})^{s}_{p,r,\gamma}$ and Theorem~\ref{DSOP:thm:trace_MR} (with $\rho=2$). So it suffices to show that this mapping is injective with a bounded right-inverse.

Let $A$ denote the realization of $\mathcal{A}(\,\cdot\,,D)$ in $F^{s}_{p,r,\gamma}(\mathscr{O})$ with domain $D(A) = F^{s+2}_{p,r,\gamma}(\mathscr{O}) = \{ u \in F^{s+2}_{p,r,\gamma,\Dir}(\mathscr{O}) : \mathrm{tr}_{\partial\dom}u = 0\}$. By \cite[Theorem~9.3]{Lindemulder2018_elliptic_weighted_B&F}, there exists $\lambda_{0} \in \R$ such that $0 \in \rho(\lambda_{0}+A)$ and $\lambda_{0}+A$ has a bounded $H^{\infty}$-calculus with angle $\omega_{H^{\infty}}(\lambda_{0}+A) < \frac{\pi}{2}$.
Since $F^{s}_{p,r,\gamma}(\mathscr{O})$ is a UMD space, it follows (see Section~\ref{DSOP:subsec:prelim:UMD_max-reg}) that, for every $\lambda \geq \lambda_{0}$,
\begin{equation}\label{DSOP:eq:thm:main;real_line;isom}
W^{1}_{q}(\R,v;F^{s}_{p,r,\gamma}(\mathscr{O})) \cap L_{q}(\R,v;F^{s+2}_{p,r,\gamma,\Dir}(\mathscr{O})) \longra L_{q}(\R,v;F^{s}_{p,r,\gamma}(\mathscr{O})),\, u \mapsto u'+(\lambda+A)u,
\end{equation}
defines an isomorphism of Banach spaces.
In particular, the mapping in \eqref{DSOP:eq:thm:main;real_line;isom} is injective.

Denoting by $\mathscr{T}(\lambda)$ the inverse corresponding to \eqref{DSOP:eq:thm:main;real_line;isom},
by $\ext_{\partial\dom}$ a coretraction for $\mathrm{tr}_{\partial\dom} \in \mathcal{B}(\M^{q,p,r}_{v,\gamma,s}(\R),\B^{q,p,r}_{v,\gamma,s}(\R))$ (see Theorem~\ref{DSOP:thm:trace_MR}), we find that
\[
\mathscr{S}(\lambda): \mathbb{D}^{q,p,r}_{v,\gamma,s}(\R) \oplus \B^{q,p,r}_{v,\gamma,s}(\R) \longra
\M^{q,p,r}_{v,\gamma,s}(\R)
\]
given by
\[
\mathscr{S}(\lambda)(f,g) := \mathscr{T}(\lambda)f + \ext_{\partial\dom}g - \mathscr{T}(\lambda)(\lambda+\mathcal{A}(\,\cdot\,,D))\ext_{\partial\dom}g
\]
defines a bounded right-inverse for the mapping in \eqref{DSOP:eq:thm:main;real_line;isom}.
\end{proof}

As a consequence of the above theorem we obtain the following corresponding result, Theorem~\ref{DSOP:thm:main;half-line}, on time intervals $J=(0,T)$ with $T \in (0,\infty]$ in the case of the power weight $v=v_{\mu}$ (with $\mu \in (-1,q-1)$), where we need to take initial values into account.

For the initial data we need to introduce the space
\[
\mathbb{I}^{q,p}_{\mu,\gamma,s} := B^{2(1+\frac{s}{2}-\frac{1+\mu}{q})}_{p,q}(\dom,w^{\partial\dom}_{\gamma}).
\]
\begin{lemma}\label{DSOP:lem:thm:main;half-line;temp_trace}
Let $\dom$ be either $\R^{d}_{+}$ or a $C^{\infty}$-domain in $\R^{d}$ with a compact boundary $\partial\dom$ and let $J=(0,T)$ with $T \in (0,\infty]$.
Let $q,p,r \in (1, \infty)$, $\mu \in (-1,q-1)$, $\gamma\in (-1,\infty)$ and $s \in (\frac{1+\gamma}{p}-2,\frac{1+\gamma}{p})$. Then $\tr_{t=0}:u \mapsto u(0)$ is a retraction
\[
\tr_{t=0}: \M^{q,p,r}_{\mu,\gamma,s}(J) \longra \mathbb{I}^{q,p}_{\mu,\gamma,s}.
\]
\end{lemma}
\begin{proof}
The corresponding result for $\dom = \R^{d}$ can be derived from \cite[Theorem~1.1]{Meyries&Veraar2014_traces}/\cite[Theorem~3.4.8]{Pruess&Simonett2016_book} using the realization of $\Delta$ on $F^{s}_{p,r}(\R^{d},w^{\partial\dom}_{\gamma})$ with domain $F^{s+2}_{p,r}(\R^{d},w^{\partial\dom}_{\gamma})$ (see \cite{Lindemulder2018_elliptic_weighted_B&F}), from which the desired result follows by an extension and restriction argument.
\end{proof}

Concerning the compatibility condition in the space of initial-boundary data $\mathbb{IB}^{q,p}_{\mu,\gamma,s}(J)$ below, let us note the following. Assume $1+\frac{s}{2}-\frac{1+\mu}{q} > \frac{1}{2}\frac{1+\gamma}{p}$.
Then, on the one hand, there is a well-defined trace operator $\mathrm{tr}_{\partial\dom}$ on $\mathbb{I}^{q,p}_{\mu,\gamma,s}(J)$; in fact, $\mathrm{tr}_{\partial\dom}$ is a retraction from $\mathbb{I}^{q,p}_{\mu,\gamma,s}$ to $B^{2(1+\frac{s}{2}-\frac{1+\mu}{q})-\frac{1+\gamma}{p}}_{p,q}(\partial\dom;X)$.
On the other hand, as a consequence of \cite[Theorem~1.1]{Meyries&Veraar2014_traces}, $\mathrm{tr}_{t=0}:g \mapsto g(0)$ is a well-defined retraction from $\mathbb{B}^{q,p,r}_{\mu,\gamma,s}(J)$ to $B^{2(1+\frac{s}{2}-\frac{1+\mu}{q})-\frac{1+\gamma}{p}}_{p,q}(\partial\dom;X)$.
Motivated by this we set
\begin{align*}
\mathbb{IB}^{q,p,r}_{\mu,\gamma,s}(J) := \left\{ (g,u_{0}) \in \mathbb{B}^{q,p,r}_{\mu,\gamma,s}(J) \oplus \mathbb{I}^{q,p}_{\mu,\gamma,s} : g(0) = \mathrm{tr}_{\partial\mathscr{O}}u_{0} \:\text{when $1+\frac{s}{2}-\frac{1+\mu}{q} > \frac{1}{2}\frac{1+\gamma}{p}$} \right\}.
\end{align*}

Now we can state the main result for the initial value problem with inhomogeneous boundary condition.
\begin{thm}\label{DSOP:thm:main;half-line}
Let $\dom$ be either $\R^{d}_{+}$ or a $C^{\infty}$-domain in $\R^{d}$ with a compact boundary $\partial\dom$ and let $J=(0,T)$ with $T \in (0,\infty]$.
Let $q,p,r \in (1, \infty)$, $\mu \in (-1,q-1)$, $\gamma\in (-1,\infty)$ and $s \in (\frac{1+\gamma}{p}-2,\frac{1+\gamma}{p})$ be such that $1+\frac{s}{2}-\frac{1+\mu}{q} \neq \frac{1}{2}\frac{1+\gamma}{p}$.
Suppose that $\mathcal{A}(\,\cdot\,,D) = \sum_{|\alpha| \leq 2}a_{\alpha}D^{\alpha}$ satisfies the smoothness condition $(\mathrm{S})^{s}_{p,r,\gamma}$ and the ellipticity condition $(\mathrm{E})$.
Then there exists $\lambda_{0} \in [-\infty,\infty)$ such that for all $\lambda \in [\lambda_{0},\infty)$,
\begin{align*}
\M^{q,p,r}_{\mu,\gamma,s}(J) \longra \mathbb{D}^{q,p,r}_{\mu,\gamma,s}(J) \oplus \mathbb{IB}^{q,p,r}_{\mu,\gamma,s}(J),\: u \mapsto (u'+(\lambda+\mathcal{A}(\,\cdot\,,D))u,\tr_{\partial\dom}u,u(0))
\end{align*}
defines an isomorphism of Banach spaces;
in particular, for all $\lambda \geq \lambda_{0}$, $f \in \mathbb{D}^{q,p,r}_{\mu,\gamma,s}(J)$ and $g \in \B^{q,p,r}_{\mu,\gamma,s}(J)$, there exists a unique solution $u \in \M^{q,p,r}_{\mu,\gamma,s}(J)$ of the parabolic initial-boundary value problem
\begin{equation*}\label{eq:thm:max-reg_bdd_domain_bd-data;interval;PBVP}
\left\{\begin{array}{rl}
u' + (\lambda+\mathcal{A}(\,\cdot\,,D)) u &= f, \\
\mathrm{tr}_{\partial\dom}u &=g, \\
u(0) &= u_{0}.
\end{array}\right.
\end{equation*}
\end{thm}

\begin{remark}\label{DSOP:rmk:thm:main;half-line;lambda0}
Remark~\ref{DSOP:rmk:thm:main;real_line;lambda0} for Theorem~\ref{DSOP:thm:main;real_line} also applies to Theorem~\ref{DSOP:thm:main;half-line}.
\end{remark}

\begin{remark}\label{DSOP:rmk:thm:main;half-line;finite_interval}
In case of a finite interval $J=(0,T)$ (so $T<\infty$) in Theorem \ref{DSOP:thm:main;half-line}, we can take $\lambda_{0}=-\infty$. This can be seen by the standard trick of considering $u_{\mu}(t):=e^{\mu t}u(t)$ instead of $u$ for suitable $\mu$.
\end{remark}

In the proof of the theorem we will use the following notation:
\[
{_{0}}\mathbb{B}^{q,p,r}_{\mu,\gamma,s}(I) := \left\{\begin{array}{ll}
\mathbb{B}^{q,p,r}_{\mu,\gamma,s}(I), & 1+\frac{s}{2}-\frac{1+\mu}{q} < \frac{1}{2}\frac{1+\gamma}{p}, \\
\{g \in \mathbb{B}^{q,p,r}_{\mu,\gamma,s}(I): g(0)=0\},& 1+\frac{s}{2}-\frac{1+\mu}{q} > \frac{1}{2}\frac{1+\gamma}{p},
\end{array}\right.
\]
and ${_{0}}\M^{q,p,r}_{\mu,\gamma,s}(I) := \{u \in \M^{q,p,r}_{\mu,\gamma,s}(I): u(0)=0\}$, where $I \in \{\R_{+},\R\}$. We will furthermore use the following lemma.

\begin{lemma}\label{lemma:thm:max-reg_bdd_domain_bd-data;interval;ext0}
Let the notation and assumptions be as in Theorem~\ref{DSOP:thm:main;half-line}.
Then operator $E_{0}$ of extension by zero from $\R_{+}$ to $\R$ is a bounded linear operator
from ${_{0}}\mathbb{B}^{q,p}_{\mu,\gamma}(\R_{+})$ to $\mathbb{B}^{q,p}_{\mu,\gamma}(\R)$.
\end{lemma}
\begin{proof}
This can be proved in the same way as \cite[Lemma~7.17]{LV2018_Dir_Laplace}.
\end{proof}

\begin{proof}[Proof of Theorem~\ref{DSOP:thm:main;half-line}]
That $u \mapsto (u'+(\lambda+\mathcal{A}(\,\cdot\,,D))u,\tr_{\partial\dom}u,u(0))$ is a bounded operator
\[
\M^{q,p,r}_{\mu,\gamma,s}(J) \longra \mathbb{D}^{q,p,r}_{\mu,\gamma,s}(J) \oplus \mathbb{B}^{q,p,r}_{\mu,\gamma,s}(J) \oplus \mathbb{I}^{q,p}_{\mu,\gamma,s}
\]
follows from a combination of the smoothness condition $(\mathrm{S})^{s}_{p,r,\gamma}$, Theorem~\ref{DSOP:thm:trace_MR} and Lemma~\ref{DSOP:lem:thm:main;half-line;temp_trace}. That it maps to $\mathbb{D}^{q,p,r}_{\mu,\gamma,s}(J) \oplus \mathbb{IB}^{q,p}_{\mu,\gamma,s}(J)$ can be seen as in \cite[Theorem~7.16]{LV2018_Dir_Laplace} follows: we only need to show that
\begin{align}\label{eq:cor:max-reg_bdd_domain_bd-data;interval;comp_cond}
\tr_{t=0}\tr_{\partial\dom}u = \tr_{\partial\dom}\tr_{t=0}u, \qquad u \in \M^{q,p,r}_{\mu,\gamma,s}(J),
\end{align}
when $1+\frac{s}{2}-\frac{1+\mu}{q} > \frac{1}{2}\frac{1+\gamma}{p}$, which simply follows from
\[
W^{1}_{q,\mu}(J;F^{s+2}_{p,q,\gamma}(\dom)) \stackrel{d}{\hookrightarrow} \M^{q,p,r}_{\mu,\gamma,s}(J).
\]
Here this density follows from a standard convolution argument (in the time variable).

Let $\lambda_{0}$ and $A$ be as in the prof of Theorem~\ref{DSOP:thm:main;real_line}.
Then, as $0 \in (\lambda_{0}+A)$ and $\lambda_{0}+A$ has a bounded $H^{\infty}$-calculus with angle $\omega_{H^{\infty}}(\lambda_{0}+A) < \frac{\pi}{2}$, $-(\lambda_{0}+A)$ is the generator of an exponentially stable $C_{0}$-semigroup (see \cite{EN}). We can now proceed as in \cite[Theorem~7.16]{LV2018_Dir_Laplace}.

Injectivity of $u \mapsto (u'+(\lambda+\mathcal{A}(\,\cdot\,,D))u,\tr_{\partial\dom}u,u(0))$ follows the fact that $-(\lambda_{0}+A)$ is the generator of a $C_{0}$-semigroup.
So it remains to be shown that it has a bounded right-inverse, i.e.\ there is a bounded solution operator to the associated parabolic initial-boundary value problem. Using Lemma~\ref{DSOP:lem:thm:main;half-line;temp_trace} followed by Theorem~\ref{DSOP:thm:trace_MR} and \eqref{eq:cor:max-reg_bdd_domain_bd-data;interval;comp_cond}, we may restrict ourselves to the case $u_{0}=0$. Furthermore, by $L_{q,\mu}$-maximal regularity of $\lambda_{0}+A$ (obtained through the $H^{\infty}$ calculus and the UMD property of $F^{s}_{p,r,\gamma}(\mathscr{O})$, see Section~\ref{DSOP:subsec:prelim:UMD_max-reg}) we may restrict ourselves to the case $f=0$.
By extension and restriction it is enough to treat the resulting problem for $J=\R_{+}$.
We must show that there is a bounded linear solution operator $\mathscr{S}:{_{0}}\mathbb{B}^{q,p,r}_{\mu,\gamma,s}(\R_{+}) \to
{_{0}}\M^{q,p,r}_{\mu,\gamma,s}(\R_{+}),\,g \mapsto u$ for the problem
\begin{equation}\label{eq:thm:max-reg_bdd_domain_bd-data;interval;1}
\left\{\begin{array}{rl}
u' + (\lambda+\mathcal{A}(\,\cdot\,,D)) u &= 0, \\
\tr_{\partial\dom}u &=g.
\end{array}\right.
\end{equation}

Let $E_{0} \in \mathcal{B}({_{0}}\mathbb{B}^{q,p}_{\mu,\gamma}(\R_{+}),\mathbb{B}^{q,p}_{\mu,\gamma}(\R))$ be the operator of extension by zero (see Lemma~\ref{lemma:thm:max-reg_bdd_domain_bd-data;interval;ext0}) and let $\mathscr{S}_{\R}:\mathbb{B}^{q,p}_{\mu,\gamma}(\R) \to
\M^{q,p}_{\mu,\gamma}(\R),\,g \mapsto u$ be the solution operator for the problem \eqref{eq:thm:max-reg_bdd_domain_bd-data;interval;1} on $\R$ from Theorem~\ref{DSOP:thm:main;real_line}.
It suffices to show that $\mathscr{S}_{\R} \circ E_{0}$ maps to ${_{0}}\mathbb{B}^{q,p}_{\mu,\gamma}(\R_{+}) $ to ${_{0}}\M^{q,p}_{\mu,\gamma}(\R)$; indeed, in that case $\mathscr{S}g:=(\mathscr{S}E_{0}g)_{|\R_{+}}$ is as desired. This can be done as in \cite[Theorem~7.16]{LV2018_Dir_Laplace}, using exponential stability of the $C_{0}$-semigroup generated by $-(\lambda_{0}+A)$.
\end{proof}

\subsection{Smoothing in the $L_{q,\mu}$-$F^{s}_{p,r,\gamma}$-maximal Regularity Approach}

We get the following two smoothing results as collaries to Theorems \ref{DSOP:thm:main;real_line} and \ref{DSOP:thm:main;half-line}, respectively, through the embedding \eqref{DSOP:eq:prelim:emb_F_into_Lp} and \eqref{DSOP:eq:prelim:Sob_embd}.

\begin{cor}\label{DSOP:cor:thm:main;real_line;smoothing}
Let $\dom$ be either $\R^{d}_{+}$ or a $C^{\infty}$-domain in $\R^{d}$ with a compact boundary $\partial\dom$.
Let $q,p,r \in (1, \infty)$, $v\in A_q(\R)$, $\gamma\in (-1,\infty)$ and $s \in (\frac{1+\gamma}{p}-2,\frac{1+\gamma}{p})$.
Suppose that $\mathcal{A}(\,\cdot\,,D) = \sum_{|\alpha| \leq 2}a_{\alpha}D^{\alpha}$ has bounded $C^{\infty}$-coefficients with $a_{\alpha}(\infty) := \lim_{|x| \to \infty}a_{\alpha}(x)$ for each $|\alpha|=2$ and that it satisfies the ellipticity condition $(\mathrm{E})$.
Then there exists $\lambda_{0} \in \R$ such that for all $\lambda \geq \lambda_{0}$,
\begin{align*}
&\left\{ u \in W^{1}_{q}(\R,v;F^{s}_{p,r,\gamma}(\dom))\cap L_{q}(\R,v;F^{s+2}_{p,r,\gamma}(\dom)) : \partial_{t}u+(\lambda+\mathcal{A}(\,\cdot\,,D))u = 0 \right\} \\
&\qquad \hookrightarrow
\quad \bigcap_{\mu > -1}\left[ W^{1}_{q}(\R,v;F^{s+\frac{\mu-\gamma}{p}}_{p,1,\mu}(\dom))\cap L_{q}(\R,v;F^{s+\frac{\mu-\gamma}{p}+2}_{p,1,\mu}(\dom)) \right] \\
&\qquad\qquad \hookrightarrow \quad \bigcap_{k \in \N}\left[ W^{1}_{q}(\R,v;W^{k}_{p}(\dom,w_{\gamma+(k-s)p}^{\partial\dom}))\cap L_{q}(\R,v;W^{k+2}_{p}(\dom,w_{\gamma+(k-s)p}^{\partial\dom})) \right].
\end{align*}
\end{cor}

\begin{cor}\label{DSOP:cor:thm:main;half-line;smoothing}
Let $\dom$ be either $\R^{d}_{+}$ or a $C^{\infty}$-domain in $\R^{d}$ with a compact boundary $\partial\dom$ and let $J=(0,T)$ with $T \in (0,\infty]$.
Let $q,p,r \in (1, \infty)$, $\mu \in (-1,q-1)$, $\gamma\in (-1,\infty)$ and $s \in (\frac{1+\gamma}{p}-2,\frac{1+\gamma}{p})$ be such that $1+\frac{s}{2}-\frac{1+\mu}{q} \neq \frac{1}{2}\frac{1+\gamma}{p}$.
Suppose that $\mathcal{A}(\,\cdot\,,D) = \sum_{|\alpha| \leq 2}a_{\alpha}D^{\alpha}$ has bounded $C^{\infty}$-coefficients with $a_{\alpha}(\infty) := \lim_{|x| \to \infty}a_{\alpha}(x)$ for each $|\alpha|=2$ and that it satisfies the ellipticity condition $(\mathrm{E})$.
Then there exists $\lambda_{0} \in [-\infty,\infty)$ such that for all $\lambda \in [\lambda_{0},\infty)$,
\begin{align*}
&\left\{ u \in W^{1}_{q,\mu}(J;F^{s}_{p,r,\gamma}(\dom))\cap L_{q,\mu}(J;F^{s+2}_{p,r,\gamma}(\dom)) : \partial_{t}u+(\lambda+\mathcal{A}(\,\cdot\,,D))u = 0, u(0)=0 \right\} \\
&\qquad \hookrightarrow
\quad \bigcap_{\mu > -1}\left[ W^{1}_{q,\mu}(J;F^{s+\frac{\mu-\gamma}{p}}_{p,1,\mu}(\dom))\cap L_{q,\mu}(J;F^{s+\frac{\mu-\gamma}{p}+2}_{p,1,\mu}(\dom)) \right] \\
&\qquad\qquad \hookrightarrow \quad \bigcap_{k \in \N}\left[ W^{1}_{q,\mu}(J;W^{k}_{p}(\dom,w_{\gamma+(k-s)p}^{\partial\dom}))\cap L_{\mu}(J;W^{k+2}_{p}(\dom,w_{\gamma+(k-s)p}^{\partial\dom})) \right].
\end{align*}
\end{cor}

\begin{remark}\label{DSOP:rmk:cor:thm:main;half-line;smoothing}\
\begin{itemize}
\item Remark~\ref{DSOP:rmk:thm:main;real_line;lambda0} for Theorem~\ref{DSOP:thm:main;real_line} also applies to Corollaries \ref{DSOP:cor:thm:main;real_line;smoothing} and \ref{DSOP:cor:thm:main;half-line;smoothing}.
\item As in Remark~\ref{DSOP:rmk:thm:main;half-line;finite_interval},
in case of a finite interval $J=(0,T)$ (so $T<\infty$), we can take $\lambda_{0}=-\infty$ in Corollary~\ref{DSOP:cor:thm:main;half-line;smoothing}.
\end{itemize}
\end{remark}

\begin{proof}[Proof of Corollaries \ref{DSOP:cor:thm:main;real_line;smoothing} and \ref{DSOP:cor:thm:main;half-line;smoothing}]
Corollaries \ref{DSOP:cor:thm:main;real_line;smoothing} and \ref{DSOP:cor:thm:main;half-line;smoothing} can be derived from Theorems \ref{DSOP:thm:main;real_line} and \ref{DSOP:thm:main;half-line}, respectively, through the embedding \eqref{DSOP:eq:prelim:emb_F_into_Lp} and \eqref{DSOP:eq:prelim:Sob_embd} in the same way. Let us only write down the argument for Corollary~\ref{DSOP:cor:thm:main;real_line;smoothing}.

Note that the smoothness condition $(\mathrm{S})^{s+\frac{\mu-\gamma}{p}}_{p,r,\mu}$ is satisfied for each $\mu \in (-1,\infty)$. So by the independence of $\lambda_{0}$ mentioned in Remark~\ref{DSOP:rmk:thm:main;real_line;lambda0}, we can choose $\lambda_{0} \in \R$ in Theorem~ \ref{DSOP:thm:main;real_line} that works simultaneously for all smoothness parameters $s_{\mu}:= s+\frac{\mu-\gamma}{p}$ and weight parameters $\mu \in (-1,\infty)$.
So pick such a $\lambda_{0}$ and let $\lambda \geq \lambda_{0}$.
Then
\begin{equation*}
\left\{u \in \M^{q,p,r}_{v,\mu,s_{\mu}}(\R):u'+(\lambda+\mathcal{A}(\,\cdot\,,D))u=0 \right\} \longra \B^{q,p,r}_{v,\mu,s_{\mu}}(\R),\,u \mapsto \tr_{\partial\dom}u
\end{equation*}
defines an isomorphism of Banach spaces, while
\[
\B^{q,p,r}_{v,\mu,s_{\mu}}(\R) = \B^{q,p,r}_{v,\gamma,s}(\R) \qquad \text{as} \qquad 1+\frac{s}{2}-\frac{1}{2}\frac{1+\gamma}{p} = 1+\frac{s_{\mu}}{2}-\frac{1}{2}\frac{1+\mu}{p}.
\]
As a consequence,
\[
\left\{u \in \M^{q,p,r}_{v,\gamma,s}(\R):u'+(\lambda+\mathcal{A}(\,\cdot\,,D))u=0 \right\}
\hookrightarrow \left\{u \in \M^{q,p,r}_{v,\mu,s_{\mu}}(\R):u'+(\lambda+\mathcal{A}(\,\cdot\,,D))u=0 \right\}
\]
for all $\mu \in (-1,\infty)$. Combining this with the microscopic improvement in the Sobolev embedding \eqref{DSOP:eq:prelim:Sob_embd}, we obtain the obtain the first inclusion to be proved.
The second inclusion follows immediately from the elementary embedding \eqref{DSOP:eq:prelim:emb_F_into_Lp}.
\end{proof}

\def\cprime{$'$} \def\cprime{$'$} \def\cprime{$'$}

\end{document}